\documentclass{amsart}
\usepackage{amssymb,mathrsfs}
\usepackage{color}

\newcommand{\rr}{}
\newcommand{\rb}
{}
\newcommand{\bcbb}{}

\def\bC {\mathbf{C}}

\def\bN {\mathbf{N}}

\def\bR {\mathbf{R}}

\def\fH {\mathfrak{H}}
\def\fS {\mathfrak{S}}

\def\cC {\mathcal{C}}

\def\cL {\mathcal{L}}

\def\cR {\mathcal{R}}
\def\cS {\mathcal{S}}
\def\cT {\mathcal{T}}

\def\a {{\alpha}}
\def\b {{\beta}}
\def\g {{\gamma}}

\def\de {{\delta}}
\def\eps {{\epsilon}}

\def\si {{\sigma}}
\def\Si {{\Sigma}}

\def\d {{\partial}}
\def\grad {{\nabla}}
\def\Dlt {{\Delta}}
\def\h{{\hbar}}

\def\rstr {{\big |}}

\def\la {\langle}
\def\ra {\rangle}

\def\nPh{{ {\norm\Phi_{L^\infty}} }}

\def\vph{{ \varphi}}

\newcommand{\Tr}{\operatorname{trace}}

\newcommand{\Lip}{\operatorname{Lip}}

\newcommand{\MKo}{\operatorname{dist_{1}}}

\newcommand{\MKd}{\operatorname{dist_{MK,2}}}
\newcommand{\Op}{\operatorname{OP}}

\newcommand{\ba}{\begin{aligned}}
\newcommand{\ea}{\end{aligned}}

\newcommand{\dys}{{g}}

\newcommand{\be}{\begin{equation}}
\newcommand{\ee}{\end{equation}}

\newcommand{\lb}{\label}

\newtheorem{Thm}{Theorem}[section]
\newtheorem{Rmk}[Thm]{Remark}

\newtheorem{Cor}[Thm]{Corollary}
\newtheorem{Lem}[Thm]{Lemma}

\newcommand{\bcb }{}
\newcommand{\ec }{ }

\newcommand{\bcr }{}
\newcommand{\ecr }{}

\newcommand{\norm}[1]{\lVert #1 \rVert}

\newcommand{\un} [1]{\underline #1}

\usepackage{scalerel}
\usepackage{stackengine,wasysym}

\newcommand\reallywidetilde[1]{\ThisStyle{%
  \setbox0=\hbox{$\SavedStyle#1$}%
  \stackengine{-.1\LMpt}{$\SavedStyle#1$}{%
    \stretchto{\scaleto{\SavedStyle\mkern.2mu\AC}{.5150\wd0}}{.6\ht0}%
  }{O}{c}{F}{T}{S}%
}}

\renewcommand{\widetilde}{\reallywidetilde}

\begin{document}

\title[From $N$-Body Schr\"odinger to Hartree: Uniformity in $\hbar$]{On the Derivation of the Hartree Equation\\ from the $N$-Body Schr\"odinger Equation: Uniformity in the Planck Constant}

\author[F. Golse]{Fran\c cois Golse}
\address[F.G.]{CMLS, Ecole polytechnique, CNRS, Universit\'e Paris-Saclay, 91128 Palaiseau Cedex, France}
\email{francois.golse@polytechnique.edu}

\author[T. Paul]{Thierry Paul}
\address[T.P.]{CMLS, Ecole polytechnique, CNRS, Universit\'e Paris-Saclay, 91128 Palaiseau Cedex, France}
\email{thierry.paul@polytechnique.edu}

\author[M. Pulvirenti]{Mario Pulvirenti}
\address[M.P.]{Sapienza Universit\`a di Roma, Dipartimento di Matematica, Piazzale Aldo Moro 5, 00185 Roma,  and International Research Center on the Mathematics and
Mechanics of Complex Systems M$\&$MoCS,
University of L'Aquila, Italy}
\email{pulvirenti@mat.uniroma1.it} 

\begin{abstract}
In this paper the Hartree equation is derived from the $N$-body Schr\"odinger equation in the mean-field limit, with convergence rate estimates that are uniform in the Planck constant $\hbar$. Specifically, we consider the two following cases:
(a) T\"oplitz initial data and Lipschitz interaction forces, and (b) analytic initial data and interaction potential, over short  time intervals independent of $\hbar$. The convergence rates in these two cases are $1/\sqrt{\log\log N}$ and $1/N$ 
respectively. The treatment of the second case is entirely self-contained and all the constants appearing in the final estimate are explicit. It provides a derivation of the Vlasov equation from the $N$-body classical dynamics using BBGKY 
hierarchies instead of empirical measures.
\end{abstract}



\date{}

\maketitle

\tableofcontents


\section{Introduction}\label{intro}


We consider the evolution of a system of $N$ quantum particles interacting through a (real-valued) two-body, even potential $\Phi$, described for any value of the Planck constant $\hbar>0$ by the Schr\"odinger equation 
$$
i\hbar\partial_t\psi=H_N\psi\,,\qquad\psi\rstr_{t=0}=\psi_{in}\in\fH_N:=L^2(\bR^d)^{\otimes N}\,,
$$
where 
$$
H_N:=-\tfrac12\hbar^2\sum_{k=1}^N\Delta_{x_k}+\frac1{2N}\sum_{1\leq k,l\leq N}\Phi(x_k-x_l)
$$
is the $N$-body Hamiltonian. With the notation
$$
\un x_N:=(x_1,\ldots,x_N)\in(\bR^d)^N\,,
$$
the $N$-body Hamiltonian is recast as
$$
H_N={-}\tfrac12\hbar^2\Dlt_{X_N}+V_N(\un x_N)\,,
$$
where $V_N$ is the $N$-body potential in the mean-field scaling, i.e. with $1/N$ coupling constant:
$$
V_N(\un x_N):=\frac1{2N}\sum_{1\leq k,l\leq N}\Phi(x_k-x_l)\,.
$$

Instead of the Schr\"odinger equation written in terms of wave functions, we shall rather consider the quantum evolution of density matrices. An $N$-body density matrix is an operator $D_N$ such that
$$
0\le D_N=D_N^*\in\cL(\fH_N)\,,\quad\Tr_{\fH_N}(D_N)=1\,.
$$
(We denote by $\cL(\fH_N)$ the set of  bounded linear operators on $\fH_N$.)

The evolution of the density matrix of a $N$-particle system is governed for any value of the Planck constant $\hbar>0$ by the von Neumann equation
 \be\label{Ns}
\partial_tD^N=\frac1{i\hbar}[H_N,D^N]\,,\qquad D_N\rstr_{t=0}=D_N^{in}\,.
\ee
The density matrix $D_N$ for a $N$-particle system described in terms of the $N$-particle wave function $\Psi_N\equiv\Psi_N(t,\un x_N)$ is the orthogonal projection in $\fH_N$ on the one-dimensional subspace $\bC\Psi_N$. In other 
words, $D_N(t)$ is the integral operator on $\fH_N$ with integral kernel $\Psi_N(t,\un x_N)\overline{\Psi_N(t,\un y_N)}$. Up to multiplying the $N$-body wave function by a global phase factor, both formulations of the quantum dynamics 
of an $N$-particle system are equivalent.

If $D_{in}^N$ is factorized, that is of the form 
$$
D^N_{in}=D_{in}^{\otimes N}
$$
where $D_{in}$ is a density matrix on $\fH:=L^2(\bR^d)$, then $D^N(t)$ is in general \textit{not }factorized for $t>0$. However, it is known that $D_N(t)$ has a ``tendency to become factorized'' for all $t>0$ as $N\to\infty$, i.e. in the 
mean-field limit, \textit{for each} $\hbar>0$. 

The precise formulation of the ``tendency to become factorized'' involves the notion of marginal of a density operator. For each $j=1,\ldots,N$, the $j$-particle marginal $D^N_j(t)$ of $D_N(t)$ is the unique operator on $\fH_j$ such that
$$
\Tr_{\fH_N}[D^N(t)(A_1\otimes\dots\otimes A_j\otimes I_{\fH_{N-j}})]=\Tr_{\fH_j}[D^N_j(t)(A_1\otimes\dots\otimes  A_j)]\,.
$$
for all $A_1,\ldots,A_j\in\cL(\fH)$. In terms of this notion of $j$-particle marginal, the ``tendency of $D_N(t)$ to become factorized'' is expressed as follows. For each $\hbar>0$ and each $t>0$, one has 

$$
D^N_j(t)\to(D^{H}(t))^{\otimes j}\quad\hbox{ for all }j\ge 1\hbox{ as }N\to\infty
$$
(in some appropriate topology on the algebra of operators on $\fH_j$), where $D^H$ is the solution of the Hartree equation
\be\label{hartree}
i\hbar\partial_t D^{H}(t)=\left[-\tfrac12\hbar^2\Delta_{\bR^d}+\Phi_{\rho^{H}(t)},D^{H}(t)\right]\,,\qquad D^{H}\rstr_{t=0}=D_{in}\,.
\ee
Here 
$$
\Phi_{\rho^{H}(t)}=\Phi\star\rho^{H}(t,\cdot)\,,
$$
denoting by $\star$ the convolution on $\bR^d$, where 
$$
\rho^H(t,x):=D^H(t,x,x)\,.
$$
In this formula, and throughout this paper, we abuse the notation and designate by $D^H(t,x,y)$ the integral kernel of the operator ${D^{H}(t)}$ on $\fH:=L^2(\bR^d)$.

This program has been carried out in a sequence of papers: see \cite{Spohn1,BGM}  for bounded interaction potentials $\Phi$, and in \cite{EY,Pickl1} for interaction potentials that have a singularity at the origin, such as the Coulomb potential 
(see also \cite{BEGMY}). See \cite{BPS} for a more detailed discussion of this subject, supplemented with an extensive bibliography. In all these results the convergence rate as $N\to\infty$ deteriorates as $\hbar\to 0$.

\smallskip
The analogous result in the classical setting, where \eqref{Ns} and \eqref{hartree} are replaced respectively by the Liouville and the Vlasov equations, has been known for a long time:  see \cite{NeunWick, BraunHepp, Dobru}, where this limit 
has been established by different methods.

Since the Liouville and Vlasov equations are the classical limits of the Schr\"odinger and the Hartree equations respectively (see e.g. \cite{LionsPaul} for precise statements), this suggests that the mean-field limit in quantum dynamics might
hold uniformly in $\hbar$, at least for some appropriate class of solutions. 

This problem has been addressed in \cite {PP}, where the mean-field limit is established by means of a semiclassical expansion for which the term-by-term convergence can be established rigorously. The limit of the $N$-body quantum dynamics 
as $N\to\infty$ and $\hbar\to 0$ jointly, leading to the Vlasov equation, has been discussed in \cite{GMP}. (See also \cite{FGS} for a uniform in $\hbar$ estimate of some appropriate distance between the $N$-body and the Hartree dynamics, 
which applies to velocity dependent interaction potentials only.) The case of fermions at $\hbar=1$, leading to an effective Planck constant $\hbar\sim N^{-1/3}$ after some convenient rescaling of the time variable, has been first discussed in \cite{EESY} (see also \cite{NarnhoSewell}).

More recently, a new approach based on the quantization of the quadratic Monge-Kantorovich (or Wasserstein) distance, analogous to the one used in \cite {Dobru}, has been introduced in \cite {GMouPaul,GP}. It provides an estimate of the 
convergence rate in the mean-field limit ($N\to\infty$) of the $N$-body quantum dynamics that is uniform in $\hbar$ as $\hbar\to 0$.

\smallskip
In the present paper, we complete the results recalled above on the mean-field limit of the $N$-body quantum dynamics leading to the Hartree equation with two different theorems. Both statements estimate some distance between the solutions
of \eqref{Ns} and of \eqref{hartree} as $N\to\infty$ uniformly in $\hbar\in[0,1]$, without assuming that $\hbar\to 0$.

First, by some kind of interpolation between the convergence rate obtained in \cite{GMouPaul} and the ``standard'' convergence rate from \cite{Spohn1,BGM} for $\hbar>0$ fixed, we prove a $O(1/\sqrt{\log\log N})$ convergence rate for the 
mean-field ($N\to\infty$) limit of the quantum $N$-body problem leading to the Hartree equation, uniformly in $\hbar\in[0,1]$ (Theorem \ref{mainlog}). This estimate is formulated in terms of a Monge-Kantorovich type distance analogous to 
Dobrushin's in \cite{Dobru}, and holds for initial data that are T\"oplitz operators. (The notion of T\"oplitz operator is recalled in the next section).

Our second result (Theorem \ref{main}) will establish the same convergence in a much stronger topology using \textit{analytic norms} together with the formalism of Wigner functions, a tool particularly well adapted to the transition from the 
quantum to the classical dynamics. Using this approach requires strong analytic assumptions on both the potential and the initial data. The advantage of this result over the previous one is a faster convergence rate, of order $O(1/N)$, which 
is much more satisfying than $O(1/\sqrt{\log\log N})$ from the physical point of view. Indeed, keeping in mind that the total number of nucleons in the universe is estimated to be of the order of $10^{80}$, an $O(1/\sqrt{\log\log N})$ convergence 
rate may be satisfying from the mathematical point of view, but is of little practical interest. Finally, as a by-product of our second result, we obtain a derivation of the Vlasov equation from the $N$-body Liouville equation exclusively based on hierarchy techniques (Theorem \ref{mainclass}).

\bigskip
The paper is organized as follows. The main results of this article, Theorems \ref{mainlog}, \ref{main} and \ref{mainclass} are presented in section \ref{result}. Section \ref{BBGKY} recalls some of the fundamental notions used in this paper 
(such as the BBGKY hierarchy, the notions of Wigner and Husimi transforms and T\"oplitz quantization). The proofs are self-contained (with the exception of the main result of \cite{GMouPaul} used in the proof of Theorem \ref{mainlog}). 
More precisely, the proofs of Theorem \ref{mainlog} is given in Part \ref{interpolation}, while the proofs of Theorems \ref{main} and \ref{mainclass} are given in Part \ref{analytic}. We have chosen this order of presentation since part of the 
material in the proof of Theorem \ref{main} is used in the proof of Theorem \ref{mainlog}.

\newpage

\section{Quantum hierarchies, T\"oplitz operators,\\  and Wigner and Husimi functions}\label{BBGKY}


\subsection{The BBGKY hierarchy.} 

First we recall the formalism of BBGKY hierarchies in the quantum setting. 

\smallskip 
Let   $\fS_N$ be the group of permutations of the set $\{1,\dots,N\}$. For each $\si\in\fS_N$ and each $\un x_N=(x_1,\ldots,x_N)\in(\bR^d)^N$, we denote
$$
U_\si\Psi_N(\un x_N):=\Psi(\si\cdot \un x_N)\,,\quad\hbox{ with }\si\cdot \un x_N:=(x_{\si^{-1}(1)},\ldots,x_{\si^{-1}(N)})\,.
$$
Everywhere in this paper, it is assumed that the $N$ particles under consideration are indistinguishable, meaning that, for each $t\ge 0$, one has
\be\lb{Indist}
U_\si D^N(t)U_\si^*=D^N(t)\,,\quad\hbox{ for all }\si\in\fS_N\,.
\ee
A straightforward computation shows that this condition is verified provided that
$$
U_\si D^N_{in}U_\si^*=D^N_{in}\,,\quad\hbox{ for all }\si\in\fS_N\,.
$$
 
Multiplying both sides of \eqref{Ns} by $B_j\otimes I_{\fH_{N-j}}$, with $B_j\in\cL(\fH_j)$ such that $[\Dlt_{(\bR^d)^j},B_j]\in\cL(\fH_j)$, one arrives at the following system of coupled equations satisfied by the sequence of marginals $D_j^N(t)$ 
of the $N$-particle density $D^N(t)$:
\be\label{NShier}
i\hbar\partial_tD^N_j=[-\tfrac12\hbar^2\Delta_{(\bR^d)^j},D^N_j]+\frac1N \cT_jD^N_j+\frac{N-j}N\cC_{j+1}D^N_{j+1},  
\ee
where
\be\label{STj}
\cT_jD^N_j=\left[\tfrac12\sum_{l\neq r=1}^j\Phi(x_l-x_r),D^N_j\right]\,,
\ee
and
\be\label{SCj}
\cC_{j+1}D_{j+1}^N=\left(\left[\sum_{l=1}^j\Phi(x_l-x_{j+1}),D_{j+1}^N\right]\right)_j\,,
\ee
and where the subscript $j$ on the right hand side of the last equality designates the $j$th marginal as in Section \ref{intro}. Everywhere in this paper, we set
$$
D^N_j=0\quad\hbox{ for  all }j>N\,.
$$
In the limit as $N\to\infty$ and for each $j$ kept fixed, one expects that $D^N_j$ converges in some sense to $D_j$ satisfying the infinite sequence of equations
\be\lb{Hhier}
i\hbar\partial_tD_j=[-\tfrac12\hbar^2\Delta_{(\bR^d)^j},D_j]+\cC_{j+1}D_{j+1}\,,\quad j\ge 1\,.
\ee
This inifinite hierarchy of equations is the ``formal limit'' of \eqref{NShier}, for each $j\ge 1$ fixed in the limit as $N\to\infty$.

An elementary, but fundamental observation is the following fact:
$$
F(t)\hbox{ is a solution of \eqref{hartree}}\Rightarrow \{F(t)^{\otimes j}\}_{j\ge 1}\hbox{ is a solution of (\ref{Hhier})}\,.
$$

\subsection{The Wigner function.} 

To each trace-class operator $F$ defined on $\fH_N=L^2((\bR^d)^N)$ with integral kernel $F(\un x_N,\un y_N)$, we associate its Wigner function $W_\hbar[F]$ defined on the underlying phase space $(T^\star\bR^d)^N=(\bR^d\times\bR^d)^N$ 
by the formula
\be\label{WTransf}
W_\hbar[F](\un x_N;\un v_N):=\tfrac1{(2\pi)^{dN}}\int_{(\bR^d)^N}F(\un x_N+\tfrac12\hbar \un y_N,\un x_N-\tfrac12\hbar \un y_N)e^{-i\un v_N\cdot \un y_N}d\un y_N\,.
\ee
Notice that $W_\hbar[F]$ is well defined for each trace-class operator $F$ and for each $\hbar>0$, since
$$
W_\hbar[F](\un x_N;\un v_N)=\frac1{(\pi\hbar)^{Nd}}\Tr[FM(\un x_N,\un v_N)]\,,
$$ 
where $M(\un x_N,\un v_N)$ is the unitary operator on $L^2((\bR^d)^N)$ defined  by 
$$
M(\un x_N,\un v_N):\,\,\phi(X)\mapsto\phi(2\un x_N-X)e^{2i\un v_N\cdot(X-\un x_N)/\hbar}\,.
$$
By Fourier inversion theorem, one has
$$
F(\un x_N,\un y_N)=\int_{(\bR^d)^N}W_\hbar[F](\tfrac12(\un x_N+\un y_N);\un v_N)e^{i\un v_N\cdot(\un x_N-\un y_N)/\hbar}d\un v_N
$$
for a.e. $\un x_N,\un y_N$, provided that $W_\hbar[F](\tfrac12(\un x_N+\un y_N),\cdot)\in L^1((\bR^d)^N)$. If this condition is not satisfied, the right hand side in the identity above should be understood as the (inverse) Fourier transform of a
tempered distribution. In particular 
$$
F(\un x_N,\un x_N)=\int_{(\bR^d)^N}W_\hbar[F](\un x_N;\un v_N)d\un v_N\,,
$$
so that 
$$
\Tr(F)=\int_{(\bR^d\times\bR^d)^N}W_\hbar[F](\un x_N;\un v_N)d\un v_Nd\un x_N\,.
$$
More generally, one can check that 
$$
\Tr(F_1F_2)=(2\pi\hbar)^{dN}\int_{(\bR^d\times\bR^d)^N}W_\hbar[F_1](\un x_N;\un v_N)W_\hbar[F_2](\un x_N;\un v_N)d\un v_Nd\un x_N\,.
$$
From this identity, one easily deduces that
$$
\ba
W_\hbar[F_j](\un x_j;\un v_j)&=\int_{(\bR^d\times\bR^d)^{N-j}} W_\hbar[F](\un x_N;\un v_N)dx_{j+1}dv_{j+1}\ldots dx_Ndv_N
\\
&=(W_\hbar[F])_j(\un x_j;\un v_j)
\ea
$$
where $\un x_j=(x_1,\dots,x_j)$ and $\un v_j=(v_1,\dots,v_j)$, and denoting by $f_j$ the $j$th marginal of the probability density or measure $f$.

A straightforward computation shows that $t\mapsto D^N(t)$ is a solution of \eqref{Ns} if and only if $t\mapsto W_\hbar[D^N](t,\cdot,\cdot)=:f^N(t)$ is a solution of the following equation, referred to as the ``Wigner equation''
\be\label{Nw}
\ba
(\partial_t+\un v_N\cdot\nabla_{\un x_N})f^N(t,\un x_N;\un v_N)&
\\
=\!\!\int_{(\bR^d)^N}\!\widehat {V_N}(\un z_N)e^{i\un z_N\cdot \un x_N}\frac{f^N(t,\un x_N;\un v_N\!+\!\tfrac12\hbar \un z_N)\!-\!f^N(t,\un x_N;\un v_N\!-\!\tfrac12\hbar \un z_N)}{i\hbar}d\un z_N&\,,
\ea
\ee
where $\widehat {V_N}$ designates the Fourier transform of $V_N$, normalized as follows
$$
\widehat {V_N}(\un v_N)=\frac1{(2\pi)^{dN}}\int_{(\bR^d)^N}V_N(\un x_N)e^{-i\un v_N\cdot\un x_N}d\un x_N\,.
$$
In the limit as $\hbar\to 0$, the Wigner equation reduces (formally) to the Liouville equation, since
$$
\ba
\int_{(\bR^d)^N}\widehat {V_N}(\un z_N)e^{i\un z_N\cdot\un x_N}\frac{f^N(t,\un x_N;\un v_N+\tfrac12\hbar\un z_N)-f^N(t,\un x_N;\un v_N-\tfrac12\hbar\un z_N)}{i\hbar}d\un z_N&
\\
\to-i\int_{(\bR^d)^N}\widehat {V_N}(\un z_N)e^{i\un z_N\cdot\un x_N}Z_N\cdot\grad_{\un v_N}f^N(t,\un x_N;\un v_N)d\un z_N&
\\
=\grad V_N(\un x_N)\cdot\grad_{\un v_N}f^N(t,\un x_N;\un v_N)&\,.
\ea
$$
See Proposition II.1 in \cite{LionsPaul} for a complete proof.

\subsection{The Husimi function.} 

The Husimi transform of a density matrix $F$ on $\fH:=L^2(\bR^d)$ is defined by the formula
$$
\widetilde W_\hbar[F](q;p):=(2\pi\hbar)^{-d}\la p,q|F|p,q\ra\ge 0\,,
$$
where 
$$
|p,q\ra(x):=(\pi\hbar)^{-d/4}e^{-(x-q)^2/2\hbar}e^{ip\cdot x/\hbar}\,.
$$
Here and in the sequel, we use the Dirac bra-ket notation whenever convenient:
$$
\la p,q|F|p,q\ra:=\left( |p,q\ra\big|F|p,q\ra\right)_{L^2(\bR^d)}\,,
$$
and $|p,q\ra\la p,q|$ is the orthogonal projector on $|p,q\ra$.
 
At variance with the Wigner function, there is no explicit ``reconstruction" formula for $F$ out of $\widetilde W_\hbar[F]$. Observe that the Husimi function captures only the diagonal part of the matrix elements of $F$ on the (over)complete 
generating system
$$
\{|p,q\ra\,\,|\,\,(p,q)\in \bR^{2d}\}\,.
$$
Yet $\widetilde W_\hbar[F]$ determines all the numbers $\la p',q'|F|p,q\ra$ as $p,q,p',q'$ run through $\bR^d$, and therefore $F$ itself. Indeed, denoting $z:=q+ip$ and $z'=q'+ip'$, a straightforward computation shows that 
$$
e^{(|q|^2+|q'|^2)/2\hbar}\la p',q'|F|p,q\ra=f(\tfrac12(z+\bar{z'});\tfrac12i(\bar{z'}-z))
$$
where $f$ is an entire function on $\bC^{2d}$. Knowing $\widetilde W_\hbar[F](q,p)$ for all $q,p\in\bR^d$ is equivalent to knowing $f(\tfrac12(z+\bar{z'});\tfrac12i(\bar{z'}-z))$ for $z=z'$, which is in turn equivalent to knowing the restriction of
$f$ to $\bR^{2d}\subset\bC^{2d}$. Since $f$ is holomorphic on $\bC^{2d}$, knowing the restriction of $f$ to $\bR^{2d}$ determines $f$ uniquely. In other words, knowing of $\tilde W_\hbar[F]$ determines $F$ uniquely.

The Husimi function is sometime referred to as a mollified Wigner function, because of the following straightforward formula:
$$
\widetilde W_\hbar[F](q;p)=e^{\hbar\Delta_{q,p}/4}W_\hbar[F](q;p)\,.
$$
In particular, both the Wigner and the Husimi function of a $\hbar$-dependent family of density matrices have the same limit (if any) in the sense of distributions as $\hbar\to 0$ (see Theorem III.1 in \cite{LionsPaul}). 

\subsection{T\"oplitz quantization.} 

Finally, we recall the definition of T\"oplitz operators used in \cite{GMouPaul}. For each positive Borel measure $\mu$ on $\bR^d\times\bR^d$, the T\"oplitz operator with symbol $\mu$ is
\be\lb{DefToplitz}
\Op^T_\hbar[\mu]:=\tfrac1{(2\pi\hbar)^d}\int_{\bR^d\times\bR^d}|p,q\ra\la p,q|\,\mu(dpdq)\,.
\ee
For instance 
\be\lb{DecompId}
\Op^T_\hbar[1]=I_\fH\,,
\ee
where $1$ designates the Lebesgue measure on $\bR^d\times\bR^d$. Elementary computations show that
\be\lb{HusiTopli}
W_\hbar[\Op^T_\hbar[\mu]]=e^{\hbar\Dlt_{q,p}/4}\mu\,,\quad\widetilde W_\hbar[\Op^T_\hbar[\mu]]=e^{\hbar\Dlt_{q,p}/2}\mu\,.
\ee
In particular
$$
\Tr(\Op^T_\hbar[\mu])=\int_{\bR^d\times\bR^d}\mu(dpdq)\,.
$$
Thus, the T\"oplitz operator with symbol $(2\pi\hbar)^d\mu$ where $\mu$ is a Borel probability measure on $\bR^d\times\bR^d$ is a density matrix (i.e. a nonnegative self-adjoint operator of trace $1$) on $\fH:=L^2(\bR^d)$. In other words, 
the T\"oplitz quantization maps Borel probability measures on the phase space to density matrices on $\fH$ (up to multiplication by the factor $(2\pi\hbar)^d$).


\section{The results}\label{result}


\subsection{The case of a Lipschitz continuous interaction force field.}


Our first uniform convergence result puts little regularity constraint on the interaction potential $V$. As a result, we obtain a convergence rate in some weak topology, which is most conveniently expressed in terms of a Monge-Kantorovich, 
or Wasserstein distance, associated to the cost function $(z,z')\mapsto\min(1,|z-z'|)$ on $(\bR^d\times\bR^d)^j$ for $j\ge 1$, where $|z-z'|$ designates the Euclidean distance from $z$ to $z'$.

Specifically, let $\mu,\nu$ be Borel probability measures on $(\bR^d\times\bR^d)^j$, and let $\Pi(\mu,\nu)$  be the set of probability measures on $(\bR^d\times\bR^d)^j\times(\bR^d\times~\bR^d)^j$ whose first and second marginals are 
$\mu$ and $\nu$ respectively. In other words $\Pi(\mu,\nu)$ is the set of positive Borel measure on $(\bR^d\times\bR^d)^j$ such that
$$
\iint_{(\bR^d\times\bR^d)^j}(\phi(x)+\psi(y))\pi(dxdy)=\int_{\bR^{dj}}\phi(x)\mu(dx)+\int_{\bR^{dj}}\psi(x)\nu(dx)
$$
for all $\phi,\psi\in C_b((\bR^d)^j)$.
Define
\be\lb{DefMKo}
\MKo(\mu,\nu):=\inf_{\pi\in\Pi(\mu,\nu)}\iint_{((\bR^{d}\times\bR^{d})^j)^2}\min(1,|z-z'|)\pi(dzdz')\,.
\ee
for each pair of Borel probability measures $\mu,\nu$ on $(\bR^d\times\bR^d)^j$. 

\begin{Thm}\label{mainlog}
Assume that the interaction potential $\Phi$ is a real-valued, even $C^{1,1}$ function defined on $\bR^d$, and let $\Lambda:=3+4\Lip(\grad \Phi)^2$. 

Let $F^{in}$ be a T\"oplitz operator, and let $F$ be the solution of the Hartree equation \eqref{hartree} with initial data $F^{in}$. On the other hand, for each $N\ge 1$, let $F^N$ be the solution of \eqref{Ns} with initial data $(F^{in})^{\otimes N}$, 
and let $F^N_j$ be its $j$-particle marginal for each $j=1,\ldots,N$. 

Then, for each fixed integer $j\ge 1$ and each $T>0$, one has\footnote{The statement $a_n\lesssim b_n$ \rr{in the limit as $n\to\infty$} means that $a_n\le b_n(1+\eps_n)$ with $\eps_n\to 0$ as $n\to\infty$.}
\be\lb{lnNEstim}
\sup_{\hbar>0}\sup_{|t|\le T}\MKo(\widetilde W_\hbar[F_j^N(t)],\widetilde W_\hbar[F(t)^{\otimes j}])^2\lesssim\frac{64(\log{2})jdT\norm{\Phi}_{L^\infty}(1+e^{\Lambda T})}{\log{\log{N}}}
\ee
in the limit as $N\to\infty$. 
\end{Thm}

Notice that $\MKo$ is the distance used by Dobrushin in \cite{Dobru}. For each $\phi\in C_b((\bR^d\times\bR^d)^j)$, denote by $L_1(\phi)$ the Lipschitz constant of $\phi$ for the distance $(z,z')\mapsto\min(1,|z-z'|)$, i.e.
$$
L_1(\phi):=\sup_{z\not=z'\in\bR^{dj}}\frac{|\phi(z)-\phi(z')|}{\min(1,|z-z'|)}\,.
$$
By Monge-Kantorovich duality one has
$$
\MKo(\mu,\nu)=\sup_{L_1(\phi)\le 1}\left|\int_{(\bR^d\times\bR^d)^j}\phi(z)\mu(dz)-\int_{(\bR^d\times\bR^d)^j}\phi(z)\nu(dz)\right|\,.
$$
Hence the distance $\MKo$ induces on the set of Borel probability measures on $\bR^d\times\bR^d$ a topology that is equivalent to the restriction of the Sobolev $W^{-1,1}(\bR^d\times\bR^d)$ metric. Since density matrices on $\fH:=L^2(\bR^d)$ 
are determined uniquely by their Husimi functions as recalled in the previous section, the expression 
$$
(D_1,D_2)\mapsto\MKo(\widetilde W_\hbar[D_1],\widetilde W_\hbar[D_2])
$$ 
defines a distance on the set of density matrices on $\fH_j=L^2(\bR^{dj})$. This distance obviously depends on $\hbar$ through the Husimi functions. However, if $D_1^\hbar$ and $D_2^\hbar$ are $\hbar$-dependent density operators such that
$$
\widetilde{W}_\hbar[D_1^\hbar]\to\mu_1\hbox{ and }\widetilde{W}_\hbar[D_2^\hbar]\to\mu_2\quad\hbox{ in }\cS'((\bR^d\times\bR^d)^j)
$$
in the limit as $\hbar\to 0$, then
$$
\MKo(\widetilde W_\hbar[D_1^\hbar],\widetilde W_\hbar[D_2^\hbar])\to\MKo(\mu_1,\mu_2)\,.
$$
For that reason \eqref{lnNEstim} provides a convergence rate for the mean-field limit of the quantum $N$-body dynamics that is uniform in $\hbar$.

\subsection{The case of an analytic interaction potential.}


Our second uniform convergence is based on Cauchy-Kowalevsky type estimates on the BBGKY hierarchy. It gives a much better convergence rate over finite time intervals only, and at the expense of much more stringent conditions on the 
interaction potential $\Phi$. These time intervals depend on the size of the potential and $\hbar$-dependent initial data in convenient topologies, but have no other dependence  on the Planck constant $\hbar$. The convergence rate obtained 
in this second result is formulated in terms of the following family of norms. For each $\rho>0$ and each $f\in L^1(\bR^n\times\bR^n)$, we set
$$
\norm f_\rho:=\sup_{\xi,\eta\in\bR^n}|\widetilde f(\xi,\eta)|e^{\rho(|\xi|+|\eta|)}\,,
$$
where $\widetilde f$ is the symplectic Fourier transform of $f$
\be\label{tft}
\widetilde f(\xi;\eta):=
\int_{\bR^n\times\bR^n}e^{i(x\cdot\xi-v\cdot\eta)}f(x;v)dxdv\,.
\ee

We slightly abuse the notation, and set
\be\lb{AbusNormRho}
\norm {F}_\rho:=\norm{W_\hbar[F]}_\rho\,.
\ee
for each trace-class operator $F$ on $\fH_j:=L^2(\bR^{dj})$.

Observe that, if $F$ is a density operator on $L^2(\bR^n)$, one has $\norm{F}_\rho\ge 1$ for all $\rho>0$. Indeed 
$$
1=\Tr(F)=\widetilde{W_\hbar[F]}(0,0)\leq\sup\limits_{\xi,\eta\in\mathbb R^n}|\widetilde{W_\hbar[F]}(\xi,\eta)|\le\norm{F}_\rho
$$
for each $\rho>0$.

\begin{Thm}\label{main}
Assume that, for some $\beta'>0$, one has both
$$
\norm {F^{in}}_{\beta'}<\infty\quad\hbox{ and }\int |h||\widehat\Phi(h)|e^{\beta'|h|}dh<\infty\,.
$$
Let  $0<\beta<\beta'$ and let $T\equiv~T[\log{\norm {F^{in}}_{\beta'}}+\log{2},\beta',\beta,\Phi]$ be defined by \eqref{T} below. 

Let $F(t)$ be the solution of the Hartree equation \eqref{hartree} with initial condition $F^{in}$, and let $F^N(t)$ be the solution of the $N$-body (Schr\"odinger) von Neumann  equation \eqref{Ns} with initial condition $(F^{in})^{\otimes N}$. 

Then for all $t$ such that $|t|<T$ and $N\geq j\geq 1$,
\rb{\be\label{maineq}
\norm {F^N_{j}(t)-F(t)^{\otimes j}}_{\beta}\leq\frac{C(j,{|t|}/T)}N
{\norm{\bcr F\ecr^{in}}_{\beta'}^j}\,.
\ee
}
For all $\tau\in[0,1)$, the constant $C(j,\tau)>0$ is defined by formula \eqref{DefCjt} below, depends only on $j$ and $\tau$, and satisfies
$$
C(j,\tau)\sim \frac{j2^{j}}{(1-\tau)^2}\left(1+\exp{\frac{(j+1)\log(1/\tau)+3}{e(\log(\tau))^2}}\right)\,\,\,\hbox{ as }j\to\infty.
$$

Moreover one has 
\be\label{nextineq}
\norm {F^N_{j}(t)-F(t)^{\otimes j}}_{\beta}\to 0\quad\hbox{ as }N\to\infty
\ee
uniformly as $\frac{|t|}T$ runs through compact subsets of $(-1,1)$ and for $ j\leq N$ satisfying
\rb{
\be\label{jjj}
1\leq j\leq\frac{\log{N}
}{\log{\norm{F^{in}}_{\beta'}}+2\log{2}+1/(e\log{(T/|t|)})}
\ee
}
\end{Thm}

\bigskip
\begin{Rmk}
The choice of $\log{\norm {F^{in}}_{\beta'}}+\log{2},$ in the definition of $T$ in Theorem \ref{main} and of $\log{\norm {F^{in}}_{\beta'}}+2\log{2}$ in \eqref{jjj} was made for sake of simplicity. $\norm{F^{in}}_{\beta'}+\log{2}$ could be replaced 
by any $\alpha>\norm{F^{in}}_{\beta'}$ in the definition of $T$ and  $\norm{F^{in}}_{\beta'}+\log{2}$ by any $\alpha'>\alpha$, as can be seen from the end of the proof of Theorem \ref{main} in Section \ref{endproof} --- see in particular both
formulas \eqref{grrr} and \eqref{estijgeneral}. The same remark holds true for Theorem \ref{mainclass} below.
\end{Rmk}

\bigskip
\begin{Rmk}[\textbf{Uniformity in the Planck constant}]
Observe that Theorem \ref{main} is valid for any value of $\hbar>0$. More precisely, the dependence in $\hbar$ of the right hand sides of \eqref{maineq} and \eqref{jjj} is contained in the initial data $F^{in}$ and its norm $\norm{F^{in}}_{\beta'}$
(both explicitly and  implicitly in the definition of $T$). Therefore, if we assume that $F^{in}$ has a semiclassical behavior, for instance if
\be\label{unifnorm}
\hbox{there exists }\beta'>0\hbox{ such that }\sup_{\hbar\in(0,1]}\norm{F^{in}}_{\beta'}<\infty\,,
\ee
the conclusions of Theorem \ref{main} hold uniformly in $\hbar\in(0,1]$.

More precisely, the estimate \eqref{maineq} implies that, under assumption \eqref{unifnorm}, the $N$-particle quantum dynamics converges to the Hartree dynamics uniformly in $\hbar\in(0,1]$, in $F^{in}$ running through bounded sets of 
density operators in the norm $\|\cdot\|_{\beta'}$, for some $\beta'>0$, and in $t$ running over compact subsets of $(-T^*,T^*)$, where 
$$
T^*:=~T[\sup\limits_{\hbar\in(0,1]}\norm{F^{in}}_{\beta'}+\log{2},\beta,\beta',\Phi]>0\,.
$$ 
This last observation follows from the fact that $T(\alpha,\beta,\beta',\Phi]$ is a decreasing function of $\alpha$.
\end{Rmk}

\bigskip
Examples of initial data satisfying (\ref{unifnorm})  are the T\"oplitz operators of the form 
\be\lb{Toplin}
F^{in}=\Op^T_\hbar[(2\pi\hbar)^df^{in}]=\int_{\bR^\times\bR^d}f^{in}(x;v)\,|x,v\ra\la x,v|\,dxdv
\ee
where $f^{in}$ runs through the set of probability densities on $\bR^d\times\bR^d$ satisfying the condition $\norm{f^{in}}_{\beta'}\le\hbox{Const.}$ for some $\beta'>0$. Indeed, the first formula in \eqref{HusiTopli} implies that
$$
\widetilde{W_\hbar[F^{in}]}(\xi;\eta)=e^{-\frac\hbar4(|\xi|^2+|\eta|^2)}\widetilde f^{in}(\xi;\eta)\,,
$$
so that
$$
\norm{F^{in}}_{\beta'}=\norm{W_\hbar[F^{in}]}_{\beta'}\leq\norm{f^{in}}_{\beta'}\,.
$$

Notice that Theorem \ref{main}, when restricted to, say, $\hbar=1$, reduces to the mean-field limit of the quantum $N$-body problem with $\hbar$ fixed.

\subsection{A derivation of the Vlasov equation.}


Let $f^{in}$ be a probability density on $\bR_x^d\times\bR_v^d$ such that $\norm{f}_{\beta'}<\infty$, and let $F^{in}=\Op^T_\hbar[(2\pi\hbar)^df^{in}]$ as in (\ref{Toplin}). Let $F$ be the solution of the Hartree equation with initial data 
$F^{in}$ and let $F^N$ be the solution of (\ref{Ns}) with initial data $(F^{in})^{\otimes N}$. Observe that
$$
\ba
\norm{F_j^N(t)-F(t)^{\otimes j}}_{\beta}=&\norm{W_\hbar[F_j^N(t)]-W_\hbar[F(t)^{\otimes j}]}_{\beta}
\\
=&\norm{W_\hbar[F^N(t)]_j-W_\hbar[F(t)]^{\otimes j}}_{\beta}\,.
\ea
$$
Applying (for instance) Theorem IV.2 in \cite{LionsPaul}, we see that
$$
W_\hbar[F^N(t)]\to f^N(t)\quad\hbox{ and }W_\hbar[F(t)]\to f(t)
$$
as $\hbar\to 0$, where $f(t)$ is the solution of the Vlasov equation with initial condition $f^{in}$, while $f^N(t)$ is the solution of the $N$-body Liouville equation with initial condition $(f^{in})^{\otimes N}$. Therefore, passing to the limit 
as $\hbar\to 0$ in \eqref{maineq}, we arrive at the following result, which bears on the derivation of the Vlasov equation from the classical $N$-body dynamics.

\begin{Thm}\label{mainclass}
Assume that 
$$
\norm{f^{in}}_{\beta'}<\infty\quad\hbox{ and }\quad\int_{\bR^d}|h||\widehat\Phi(h)|e^{\beta'|h|}dh<\infty\,,
$$
for some $\beta'>0$. Let $0<\beta<\beta'$, and let $T\equiv T[\log{\norm{f^{in}}_{\beta'}}+\log{2},\b,\b',\Phi]>0$ and $C(j,\tau)$ be given by \eqref{T} and by \eqref{DefCjt} respectively.  

Let $f(t)$ be the solution of the Vlasov equation with initial condition $f^{in}$, and let $f^N(t)$ be the solution of the $N$-body Liouville equation with initial condition $(f^{in})^{\otimes N}$. 

Then, for all $t$ such that $|t|<T$, all $j\in\bN^\star$ and all $N\ge j$, one has
\rb{
\be\label{maineqclass}\nonumber
\norm {f^N_{j}(t)-f(t)^{\otimes j}}_{\beta}\leq\frac{C(j,|t|/T)}N
\norm{f^{in}}_{\beta'}^j\,.
\ee
}
In particular, 
$$
\norm {f^N_{j}(t)-f(t)^{\otimes j}}_{\beta}\to 0\quad\hbox{ as }N\to\infty
$$
uniformly in $t$ over compact subsets of $(-T,T)$ and in $1\le j\le N$ satisfying
\rb{
$$
j\le\frac{\log N
}{\norm{f^{in}}_{\beta'}+2\log{2}+1/(e\log(T/|t|))}\,.
$$
}
\end{Thm}

\smallskip
Of course, there exist other derivations of the Vlasov equation from the classical $N$-body problem (see \cite{NeunWick,BraunHepp,Dobru}, or section 3 in \cite{GMouPaul}) --- see also \cite{GP} for a derivation of the Vlasov equation 
from the quantum $N$-body problem, by a joint mean-field and semiclassical limit. These derivations put less requirements on the interaction potential --- but they do not use the formalism of hierarchies. It is interesting to observe that using
\textit{exclusively} the formalism of hierarchies for this derivation seems to require more regularity on the interaction potential than the approach based on empirical measures used in \cite{NeunWick,BraunHepp,Dobru}. The approach in 
section 3 of \cite{GMouPaul} uses neither empirical measures nor hierarchies, but requires the same regularity on the interaction potential as the approach based on the empirical measure, i.e. $C^{1,1}$ potentials. Notice however Spohn's
interesting result \cite{Spohn} on the uniqueness for the Vlasov hierarchy (see also the appendix of \cite{GMouRicci}) which also requires $C^{1,1}$ regularity on the interaction potential. 


\part{Results with analytic data}\label{analytic}


\section{Bounds on the solutions of the hierarchies}\label{convhier}


Integrating the Wigner equation \eqref{Nw}, over the $N-j$ last variables, with 
$$
V_N(\un x_N)=\frac1{2N}\sum\limits_{1\leq l\neq r\leq N}\Phi(x_l-x_r)\,,
$$
leads to the following finite hierarchy, henceforth referred to as the ``Wigner hierarchy''. It is  of course equivalent to \eqref{NShier},
\be\label{wh}
\partial_tf_j^N+\sum_{i=1}^j v_i\cdot\nabla_{x_i}f_j^N=\frac1NT_jf_j^N+\frac{N-j}NC_{j+1}f_{j+1}^N,\  
\ee
where 
$$ 
\left\{
\ba
{}&f_j^N(x_1,\dots,{\un x}_j; {\un v}_1,\dots ,{\un v}_j):=\int f^N(\un x_N;\un v_N)dx_{j+1}d{v}_{j+1}\dots dx_Nd{v}_N\ \ \ \hbox{ for } j\leq N,
\\
&f^N_j:=0\ \ \ \hbox{ for }j>N.
\ea\right.
$$
In view of the definition of the norm $\norm{\cdot}_\beta$,  it is more convenient for our purpose to express $T_j$ and $C_{j+1}$ through their action on the Fourier transforms of $f_j^N$ and $f_{j+1}^N$ as defined by \eqref{tft}: this is achieved
by Lemma \ref{oufoufouf} below.

Passing to the limit as $N\to\infty$, one arrives at the infinite hierarchy 
$$
\partial_tf_j+\sum_{i=1}^j v_i\cdot\nabla_{x_i}f_j=C_{j+1}f_{j+1},\ \ \ j\geq 1,
$$
henceforth referred to as the ``Hartree hierarchy''. This is a formal statement so far, and one of the goals in our study is to turn this formal statement into a precise convergence statement with an estimate of the convergence rate. 

From now on, we regard both the Wigner and the Hartree hierarchies as governing the evolution of infinite sequences of Wigner distributions. It is therefore natural to seek controls on these hierarchies in appropriate functional spaces
of sequences $\un f=\{f_j\}_{j\ge 1}$. These functional spaces are metrized by the following family of norms:
\be\nonumber
\norm {\un f}_{\alpha,\beta}:=\sum_{j=1}^\infty e^{-\alpha j}\norm {f_j}_\beta,\quad \alpha,\beta>0\,.
\ee

For $\beta>0$ we set
\be\label{cphi}\nonumber
C_\Phi^\beta(t):=\int_{\bR^d}|h||\widehat\Phi(h)|e^{2\beta|h|(1+|t|)}{dh}\,,\qquad C_\Phi^0:=\int_{\bR^d}|h||\widehat\Phi(h)|{dh}\,.
\ee

\begin{Thm}\label{hier}
Assume that $\norm{f^{in}}_{\beta'}<\infty$ for some $\beta'>0$. Let $\un f^N(t)$ be the solution of the $N$-body Wigner hierarchy with initial data $\un f^{in}=\{f^{in}_j\}_{j\ge 1}$ defined by
$$
\ba
{}&f^{in}_j=(f^{in})^{\otimes j},\quad &j\leq N\,,
\\
&f^{in}_j=0, &j>N\,.
\ea
$$
Then for each $\alpha>\log{\norm {f^{in}}_{\rb{\beta'}}}$ and each $\beta$ such that $0<\beta<\beta'$, there exists $T\equiv T[\alpha,\beta,\beta',\Phi]>0$ given by \eqref{T} below such that, for each $t\in[0,T)$
\be\label{eqmainhier2}
\norm{\un f^N(t)}_{\alpha,\beta}\leq\frac{\norm {\un f^{in}}_{\alpha,\beta'}}{1-\frac{|t|}T}.
\ee
\end{Thm}

\begin{proof} Our proof of Theorem \ref{hier} is based on the estimates in the next lemma, whose proof is deferred until the end of the present section.

\begin{Lem}[Propagation estimates]\label{propesti}
Let $S_j(t)$ denote the group generated by the free transport operator in the $j$ first phase-space variables
$$
-\sum_{k=1}^j\xi_i\cdot\nabla_{x_i}\,.
$$
In other words,
$$
S_j(t)f^0_j:\, (x_1,\ldots,{\un x}_j;\xi_1,\ldots,\xi_j)\mapsto f^0_j(x_1-t\xi_1,\ldots,{\un x}_j-t\xi_j;\xi_1,\ldots,\xi_j)
$$
is the value at time $t$ of the solution of the Cauchy problem
$$
\d_tf_j+\sum_{k=1}^j\xi_i\cdot\nabla_{x_i}f_j=0\,,\qquad f_j\rstr_{t=0}=f^0_j\,.
$$
Then, 
\be\label{St}
\norm{S_j(t) f_j}_\beta\leq\norm{f_j}_{\beta'}\quad\hbox{ provided that }\frac{\beta'-\beta}{\beta'}\geq |t|\,,
\ee
while
\be\label{Cj}
\norm{S_j(-t)C_{j+1}S_{j+1}(t)f_{j+1}}_\beta\leq\frac{(1+|t|)C_\Phi^0}{e(\beta'-\beta)}\norm{f_{j+1}}_{\beta'}\,,
\ee
and
\be\label{Tj}
\norm{S_j(-t)T_jS_j(t)f_{j}}_\beta\leq\frac{j(1+|t|)C_\Phi^{\beta'}(t)}{e(\beta'-\beta)}\norm{f_j}_{\beta'}\,,
\ee
whenever $\beta'>\beta$.
\end{Lem}

Define
\be\label{TNCN}
T^N\un f:=\left\{\frac{T_j}Nf_j\right\}_{j\ge 1}\quad\hbox{ and }\ \ \ C^N\un f:=\left\{\frac{N-j}N{C_{j+1}}f_{j+1}\right\}_{j\ge 1}\,,
\ee
so that \eqref{wh} reads 
\be\label{whun}\nonumber
\partial_t\un f^N+\sum_{i=1}^Nv_i\cdot\nabla_{x_i}\un f^N=(T^N+C^N)\un f^N,
\ee 
and let 
$$
S(t)\un f^N:=\{S_j(t)f^N_j\}_{j\ge 1}\,.
$$

\smallskip
As a straightforward consequence of Lemma \ref{propesti}, we arrive at the following estimates.

\begin{Cor}\label{cor}
Under the same assumptions as in Lemma \ref{propesti} and for $\beta'>\beta$, one has
\be\label{estC}
\norm{S(-t)C^NS(t)\un f}_{\alpha,\beta}\leq\frac{(1+|t|)C_\Phi^0e^\alpha}{e(\beta'-\beta)}\norm{\un f}_{\alpha,\beta'}\,,
\ee
\be\label{estT}
\norm{(S(-t)T^NS(t) \underline f)_j}_{\beta}\leq\frac {j} Ne^{\alpha j}\frac{(1+|t|)C_\Phi^{\beta'}(t)}{e(\beta'-\beta)}\norm{\un f}_{\alpha,\beta'}\,,
\ee
and 
\be\label{Tab}
\norm{S(-t)T^NS(t)\un f}_{\alpha,\beta}\leq\frac{(1+|t|)C_\Phi^{\beta'}(t)}{e(\beta'-\beta)}\norm{\un f}_{\alpha,\beta'}\,. 
\ee
\end{Cor}

\smallskip
We shall compute the norm of the solution of the Wigner hierarchy expressed by a Dyson expansion. 

For $0\leq t_1\leq t_{2}\dots\leq t_n\leq t$, consider the string of distribution functions indexed by $n\ge 0$, and defined as follows: for $n>0$,
\be\label{dysonTC}
\dys^N_n (t,t_1,\dots,t_n):=S(t)S(-t_n)(T^N+C^N)S(t_n)\ldots S(-t_1)(T^N+C^N)S(t_1)\un f^{in}\,,
\ee
while
$$
\bcb\dys^N_0\ec (t):=S(t)\un f^{in}\,,
$$
where $\un f^{in}=\{f^{in}_j\}_{j=1\ldots\infty}$ is the initial data. 

For each $K\geq1$, set
$$
\un\dys^{N,K}(t):=\sum\limits_{n=0}^K\int_0^tdt_n\int_0^{t_n}dt_{n-1}\dots\int_0^{t_{2}}dt_1\dys^N_n(t,t_1,\dots,t_n)\,.
$$
We immediately see that
$$
\partial_t\un {\dys}^{N,K}+\sum_{i=1}^Nv_i\cdot\nabla_{x_i}\un {\dys}^{N,K}=(T^N+C^N)\un {\dys}^{N,K-1}, \qquad\un {\dys}^{N,K}\rstr_{t=0}=\un f^{in}\,.
$$
The Dyson expansion is
$$
\lim\limits_{K\to\infty}\un {\dys}^{N,K}(t)=:\un {\dys}^N(t)\,,
$$
and, if 
$$
\lim\limits_{K\to\infty}(T^N+C^N)\un {\dys}^{N,K-1}=(T^N+C^N)\un {\dys}^N\,,
$$
the Dyson expansion gives the  solution $\un {f}^N$ of the $N$-body Wigner hierarchy with initial data $\un f^{in}$. The uniqueness of the solution of the $N$-Wigner hierarchy is obvious for each finite $N$, since that hierarchy is equivalent
to the $N$-body quantum problem. (Indeed, the $N$-Wigner hierarchy is derived from the quantum $N$-body problem, and conversely, the $N$-th equation in the $N$-Wigner hierarchy is precisely the $N$-body Wigner equation.)

For $\beta<\beta_0<\beta'$, let us  define 
$$
\beta_k:=\beta_0+k\frac{\beta'-\beta_0}n\,\quad k=0,\dots,n\,,
$$ 
so that 
$$
\beta_k\rstr_{k=0}=\beta_0\,,\quad\beta_n=\beta'\,,\quad\hbox{ and }\beta_{k+1}-\beta_k=\frac{\beta'-\beta_0}n\,,\quad 0\le k<n\,.
$$
Set
$$
\Si_N(t_k):=S(-t_k)(T^N+C^N)S(t_k)\,.
$$

Observe that $C^0_\Phi\leq C^{\beta}_\Phi(t)$ for all $t\in\bR$ and all $\beta>0$, and that $t\mapsto C^{\beta}_\Phi(t)$ is an increasing function of $|t|$ for each $\b>0$. Applying Corollary \ref{cor} \eqref{estC}-\eqref{Tab} shows that,
for all $|t_k|\leq t$ and $0<\g<\g'$, one has
\be\lb{EstiSigman}
\norm{\Si_N(t_k)\un f}_{\alpha,\g}\leq\frac{2(1\!+\!|t_k|)C_\Phi^{\g'}(t_k)e^\alpha}{e(\g'-\g)}\norm{\un f}_{\alpha,\g'}\leq\frac{2(1\!+\!|t|)C_\Phi^{\g'}(t)e^\alpha}{e(\g'-\g)}\norm{\un f}_{\alpha,\g'}\,.
\ee
Therefore, applying Lemma \ref{propesti} \eqref{St} with $|t|\leq \frac{\beta_0-\beta}{\beta_0}$ shows that
\be\label{cxx}
\ba
\norm{{\dys}_n(t,t_1,\dots,t_n)}_{\alpha,\beta}=&\norm{S(t)\Si_N(t_n)\Si_N(t_{n-1})\ldots\Si_N(t_1)\un f^{in}}_{\alpha,\beta}
\\
\le&\norm{\Si_N(t_n)\Si_N(t_{n-1})\ldots\Si_N(t_1)\un f^{in}}_{\alpha,\beta_0}
\\
\le&\frac{2(1\!+\!|t|)C_\Phi^{\beta_1}(t)e^\alpha}{e(\beta_1-\beta_0)}\norm{\Si_N(t_{n-1})\ldots\Si_N(t_1)\un f^{in}}_{\alpha,\beta_1}
\\
=&\frac{2(1\!+\!|t|)C_\Phi^{\beta'}(t)e^\alpha }{e(\beta'-\beta_0)}n\norm{\Si_N(t_{n-1})\ldots\Si_N(t_1)\un f^{in}}_{\alpha,\beta_1}
\\
\le&\left(\frac{2(1+|t|)C_\Phi^{\beta'}(t)e^\alpha }{e(\beta'-\beta_0)}n\right)^2\norm{\Si_N(t_{n-2})\ldots\Si_N(t_1)\un f^{in}}_{\alpha,\beta_2}
\\
\cdot&
\\
\cdot&
\\
\cdot&
\\
\le&\left(\frac{2(1+|t|)C_\Phi^{\beta'}(t)e^\alpha }{e(\beta'-\beta_0)}n\right)^n\norm{\un f^{in}}_{\alpha,\beta'}\,.
\ea
\ee

Hence
\be\lb{estigeo}
\ba
\sum\limits_{n=0}^K\left\|\int_0^t dt_1\int_0^{t_1}dt_2\dots\int_0^{t_{n-1}}dt_n{\dys}_n(t,t_1,\dots,t_n)\right\|_{\alpha,\beta}
\\
\le\sum\limits_{n=0}^K\left(\frac{2(1+|t|)C_\Phi^{\beta'}(t)e^\alpha }{e(\beta'-\beta_0)}n\right)^n\frac{|t|^n}{n!}\norm{\un f^{in}}_{\alpha,\beta'}
\\
=\sum\limits_{n=0}^K\left(\frac{2(1+|t|)|t|C_\Phi^{\beta'}(t)e^\alpha }{(\beta'-\beta_0)}\right)^n\frac{n^ne^{-n}}{n!}\norm{\un f^{in}}_{\alpha,\beta'}
\\
\le\sum_{n=0}^K\left(\frac{2(1+|t|)|t|C_\Phi^{\beta'}(t)e^\alpha }{(\beta'-\beta_0)}\right)^n\norm{\un f^{in}}_{\alpha,\beta'}
\ea
\ee
where \eqref{estigeo} follows from the fact\footnote{Since $\log$ is an increasing function, one has
$$
\log{n!}=\sum\limits_{j=2}^n\log{j}\geq\int_{1}^n\log(x)dx=\big[x\log{x}-x\big]_1^n=n\log{n}-n+1>\log(n^ne^{-n})\,.
$$}
that $n!\geq   n^ne^{-n}$, for each positive integer $n$. Hence the Dyson expansion converges uniformly in $N$ as $K\to\infty$, provided that
$$
|t|\leq\frac{\beta_0-\beta}{\beta_0}\quad\hbox{ and }\quad\frac{2(1+|t|)|t|C_\Phi^{\beta'}(t)e^\alpha }{(\beta'-\beta_0)}<1\,.
$$
By construction, the function $t\mapsto C_\Phi^{\beta'}(t)$ is continuous and increasing on $\bR_+^*$, so that there exists a unique $\beta_0\equiv\beta_0[\alpha,\beta,\beta',\Phi]$ satisfying
\be\label{betazero}
\beta<\beta_0<\beta'\quad\hbox{ and }\quad 2(1+\tfrac{\beta_0-\beta}{\beta_0})\tfrac{\beta_0-\beta}{\beta_0}C_\Phi^{\beta'}(\tfrac{\beta_0-\beta}{\beta_0})e^\alpha=\beta'-\beta_0\,.
\ee
Defining
\be\lb{T}
T[\alpha,\beta,\beta',\Phi]:=\frac{\beta_0[\alpha,\beta,\beta',\Phi]-\beta}{\beta_0[\alpha,\beta,\beta',\Phi]},
\ee
we see that
$$
|t|<T[\alpha,\beta,\beta',\Phi]\Rightarrow |t|\leq\frac{\beta_0-\beta}{\beta}\quad\hbox{ and }\quad\frac{2(1+|t|)|t|C_\Phi^{\beta'}(t)e^\alpha }{(\beta'-\beta_0)}<1.
$$
Hence the Dyson expansion converges uniformly in $N$ and $|t|\le\tau$ as $K\to\infty$ if $0<\tau<T[\alpha,\beta,\beta',\Phi]$.

Observe that the map $s\mapsto(1+s)sC_\Phi^{\beta'}(s)$ is increasing on $(0,+\infty)$ for each $\beta'>\beta$, so that $\a\mapsto T[\alpha,\beta,\beta',\Phi]$ is decreasing on $(0,+\infty)$.

Applying Corollary \ref{cor} with $t=0$, and using the geometric series estimates in \eqref{estigeo} shows that the term $(T^N+C^N)\un g^{N,K-1}$ converges to $(T^N+C^N)\un g^N$ as $K\to\infty$.

Therefore, by uniqueness of the solution of the Cauchy problem for the $N$-Wigner hierarchy, $\un g^N=\un f^N$, so that $\un {\dys}^{N,K}$ converges to $\un f^N$ as $K\to\infty$. In other words, $\un f^N$ is given by the Dyson expansion. 
In particular, the estimate \eqref{eqmainhier2} follows from \eqref{estigeo}, after noticing that 
\be\label{tbeta}
\frac{2(1+|t|)|t|C_\Phi^{\beta'}(t)e^\alpha }{(\beta'-\beta_0)}\leq \frac{|t|}T\times\frac{2(1+T)TC_\Phi^{\beta'}(T)e^\alpha }{(\beta'-\beta_0)}\leq \frac{|t|}T\,,
\ee
since $\frac{2(1+T)TC_\Phi^{\beta'}(T)e^\alpha }{(\beta'-\beta_0)}=1$ by \eqref{T} and \eqref{betazero}.
\end{proof}

\smallskip
We conclude this section with the proof of Lemma \ref{propesti}.

\begin{proof}[Proof of Lemma \ref{propesti}]
\bcbb

We shall write \eqref{Nw}  in Fourier variables, henceforth denoted $(\un\xi_N;\un\eta_N)$.  More generally, we recall the notation
$$
(\un\xi_j;\un\eta_j):=(\xi_1,\dots,\xi_j;\eta_1,\dots,\eta_j)\in\bR^{2jd}
$$
In other words, $\un\xi_j$ is the Fourier variable corresponding to ${\un x}_j$ while $\un\eta_j$ is the Fourier variable corresponding to $\un v_j$. For $h\in\bR^d$ and $ j=1,\dots,N$, we set
\be\label{defh}
\un h^r_j:=(0_{r-1},h,0_{j-r})\,,
\ee
where 
$$
0_{r-1}:=(0,\ldots,0)\in\bR^{(r-1)d}\,.
$$
 
\begin{Lem}\label{oufoufouf}
The $N$-Schr\"odinger hierarchy \eqref{NShier} is equivalent to the $N$-Wigner hierarchy
$$
\partial_tf_j^N+\sum_{i=1}^j v_i\cdot\nabla_{x_i}f_j^N=\frac1NT_jf_j^N+\frac{N-j}NC_{j+1}f_{j+1}^N\,,
$$
where $f^N_j$ is the Wigner function of $D^N_j$, while $T_j$ and $C_{j+1}$ are defined by the formulas
$$
\left\{
\ba
{}&\widetilde{T_jf_j^N}(t,\un\xi_j;\un\eta_j)=\tfrac12\sum\limits_{l\neq r=1}^j\int_{\bR^d}{dh}\widehat\Phi(h)\frac{2\sin{\frac{\hbar}2(\eta_r-\eta_l)\cdot h}}\hbar\widetilde{f^N_j}(t,\un\xi_j+\un h^r_j-\un h^l_j;\un\eta_j)\,,
\\
\\
&\widetilde{C_{j+1}f_{j+1}^N}(t,\un\xi_j;\un\eta_j)=\sum\limits_{r=1}^j\int_{\bR^d}{dh}\widehat\Phi(h)\frac{2\sin(\frac{\hbar}{2}\eta_r\!\cdot\!h)}\hbar\widetilde{f_{j+1}^N}(t,\un\xi_j+\un h^r_j,-h;\un\eta_j,0)\,,
\ea
\right.
$$
where $\widetilde\cdot$ is the Fourier transform defined by \eqref{tft} with $n=jd\mbox{ or }(j+1)d$.
\end{Lem}

\smallskip
The proof of this lemma is given in Appendix \ref{prouf}.

\smallskip
\noindent
\textbf{Remark.} One recovers the classical BBGKY hierarchy as the formal limit of the above hierarchy of Wigner equations in the limit as $\hbar\to 0$.

\bigskip
One has obviously
$$
\widetilde{S_j(t)f_j}(\un\xi_j;\un\eta_j)=\widetilde{f_j}(\un\xi_j;\un\eta_j-t\un\xi_j)\,.
$$
Therefore
$$
\ba
\widetilde{S_j(-t)T_jS_j(t)f_j^N}(\un\xi_j;\un\eta_j)&
\\
=\tfrac12\sum_{1\leq l\neq r\leq j}\int_{\bR^d}{dh}\widehat\Phi(h)\frac{2\sin{(\frac{\hbar}2(\eta_r-\eta_l+t(\xi_r-\xi_l))\cdot h)}}\hbar&
\\
\times\widetilde{f^N_j}(\un\xi_j+\un h_l-\un h_r;\un\eta_j-t(\un h_l-\un h_r))&\,,
\ea
$$
and
$$
\ba
\widetilde{S_j(-t)C_{j+1}S_{j+1}(t)f_{j+1}^N}(\un\xi_j;\un\eta_j)&
\\
=\sum_{r=1}^j\int _{\bR^d}{dh}\widehat\Phi(h)\frac{2\sin{(\frac{\hbar}2(\eta_r+t\xi_r)\cdot h)}}\hbar\widetilde{f_{j+1}^N}(\un\xi_j-\un h_r,h;\un\eta_j+t\un h_r,-th)&\,.
\ea
$$

For $\beta'>\beta$, one has
$$
\ba
\norm{S_j(t)f_j}_\beta\leq\norm{f_j}_{\beta'}\sup_{\un\xi_j,\un\eta_j}\exp\left(-\beta'\sum_{i=1}^j(|\xi_i|+|\eta_i-t\xi_i|)+\beta\sum_{i=1}^j(|\xi_i|+|\eta_i+t\xi_i|)\right)&
\\
\leq\norm{f_j}_{\beta'}\exp\left(-(\beta'-\beta)\sum_{i=1}^j(|\xi_i|+|\eta_i|)+\beta'|t|\sum_{i=1}^j|\xi_i|\right)\leq\norm{f_j}_{\beta'}&\,,
\ea
$$
provided $\frac{\beta'-\beta}{\beta'}\geq |t|$, which proves \eqref{St}. 

Moreover, still with $\beta'>\beta$, 
$$
\ba
\norm{S_j(-t)C_{j+1}S_{j+1}(t)f_{j+1}}_\beta\leq\norm{f_{j+1}}_{\beta'}\sup_{\un\xi_j,\un\eta_j}e^{\beta\sum_{i=1}^j(|\xi_i|+|\eta_i|)}&
\\
\times\sum_{r=1}^j\int_{\bR^d} {dh}|\widehat\Phi(h)||h|(|\eta_r|+|t\xi_r|)e^{-\beta'\sum_{i=1}^j(|\xi_i-\de_{r,i}h|+|\eta_i+t\de_{r,i}h|)}e^{-\beta'(|h|+|th|)}&
\\
\leq\norm{f_{j+1}}_{\beta'}(1+|t|)C_\Phi^0\sup_{\un\xi_j,\un\eta_j}\sum_{k=1}^j(|\xi_k|+|\eta_k|)e^{(\beta-\beta')\sum_{i=1}^j(|\xi_i|+|\eta_i|)}&
\\
\leq\norm{f_{j+1}}_{\beta'}\frac{(1+|t|)C_\Phi^0}{e(\beta'-\beta)}&\,.
\ea
$$
The second inequality above uses the triangle inequalities $|\eta_r+th|\geq|\eta_r|-|th|$ and $|\xi_r-h|\geq|\xi_r|-|h|$. This establishes \eqref{Cj}. 

To prove \eqref{Tj} we use the same type of argument:
$$
\ba
\norm{S_j(-t)T_jS_j(t)f_j}_\beta\leq\norm{f_{j}}_{\beta'}\sup_{\un\xi_j,\un\eta_j}e^{\beta\sum_{i=1}^j(|\xi_i|+|\eta_i|)}\tfrac12\sum_{1\le r\not=l\le j}\int_{\bR^d}{dh}|\widehat\Phi(h)||h|&
\\
\times|\eta_l-\eta_r+t(\xi_l-\xi_r)|e^{-\beta'\sum_{i=1}^j(|\xi_i+\de_{i,l}h-\de_{i,r}h|+|\eta_i-t\de_{i,l}h+t\de_{i,r}h|)}&
\\
\le\norm{f_{j}}_{\beta'}\sup_{\un\xi_j,\un\eta_j}e^{(\beta-\beta')\sum_{i=1}^j(|\xi_i|+|\eta_i|)}&
\\
\times(1+|t|)\sum_{1\le r<l\le j}(|\eta_r|+|\eta_k|+|\xi_r|+|\xi_k|)\int_{\bR^d}{dh}|\widehat\Phi(h)|h|e^{2\beta'|h|(1+|t|)}&
\\
\leq\norm{f_{j}}_{\beta'}j\frac{(1+|t|)C_\Phi^{\beta'}(t)}{e(\beta'-\beta)}&\,.
\ea
$$
\end{proof}


\section{Comparison with Hartree}\label{comphar}


We want now to compare the solution of the $N$-body Wigner  hierarchy to the solution of the infinite, Hartree hierarchy. 

The solution of the Hartree hierarchy \rb{with initial condition $\un f^{in}$  (defined in Theorem \ref{hier})
will be denoted by 
$\un f^N_H(t)$.
 It} can also be constructed exactly as in the previous section, starting from the string ${\dys}_n$ given by the formula
$$
\left\{
\ba
{}&\bcb {\dys}_n\ec(t,t_1,\dots,t_n)=S(t)S(-t_n)CS(t_n)\dots S(-t_1)CS(t_1)\un {f^{in}}\,,\qquad n>0\,,
\\
&{\dys}_0(t):=S(t)\un f^{in}\,,
\ea
\right.
$$
where
$$
C\un f:=\{C_{j+1}f_{j+1}\}_{j\ge 1}\,.
$$

Let ${\dys}^N_n$ be the string leading to the $N$-body Wigner hierarchy, that is
$$
\left\{
\ba
{}&{\dys}_n^N(t,t_1,\dots,t_n):=S(t)S(-t_n)(T^N+C^N)S(t_n)\ldots S(-t_1)(T^N+C^N)S(t_1)\un f^{in},
\\
&{\dys}^N_0(t):=S(t)\un f^{in}\, ,
\ea
\right. 
$$
where we recall that $T^N$ and $C^N$ are given  by \eqref{TNCN}.

Finally we define ${\dys}_n^{HH_N}$  by
$$
\left\{
\ba
{}&{\dys}^{HH_N}_n(t,t_1,\dots,t_n)=S(t)S(-t_n)C^NS(t_n)\dots S(-t_1)C^NS(t_1)\un {f^{in}}\,,\qquad n>0\,,
\\
&{\dys}^{HH_N}_0(t):=S(t)\un f^{in}\,.
\ea
\right.
$$
Notice that ${\dys}_0={\dys}_0^{HH_N}=\bcb {\dys}_0^{N}$. 

\subsection{Comparing ${\dys}^N_n$\ec to ${\dys}_n^{HH_N}$}\label{5.1}

For each $n_0\geq 1$, one has 
$$
\ba
\left\|\sum_{n=0}^\infty\int_0^tdt_n\int_0^{t_n}dt_{n-1}\dots\int_0^{t_{2}}dt_1\left({\dys}^N_n-{\dys}_n^{HH_N}\right)_j(t,t_1,\dots,t_n)\right\|_{\beta}&
\\
\leq\left\|\sum_{n=1}^{n_0}\int_0^tdt_n\int_0^{t_n}dt_{n-1}\dots\int_0^{t_{2}}dt_1\left({\dys}^N_n-{\dys}_n^{HH_N}\right)_j(t,t_1,\dots,t_n)\right\|_{\beta}+R(n_0)_j&\,,
\ea
$$
where $R(n_0)$ is the sum of the norms of the two remainders of the (convergent) Dyson expansions of $g^N_n$ and $g^{HH_N}_n$. 

Proceeding as in \eqref{estigeo}, one finds that the remainder for $g^N_n$ is estimated by
$$
\ba
\sum\limits_{n>n_0}\left\|\int_0^t dt_1\int_0^{t_1}dt_2\dots\int_0^{t_{n-1}}dt_n{\dys}^N_n(t,t_1,\dots,t_n)\right\|_{\alpha,\beta}&
\\
\le\sum_{n>n_0}\left(\frac{2(1+|t|)|t|C_\Phi^{\beta'}(t)e^\alpha }{(\beta'-\beta_0)}\right)^n\norm{\un f^{in}}_{\alpha,\beta'}&
\\
\le\sum_{n>n_0}\frac{|t|^n}{T^n}\norm{\un f^{in}}_{\alpha,\beta'}=\frac{(|t|/T)^{n_0+1}}{1-\frac{|t|}{T}}\norm{\un f^{in}}_{\alpha,\beta'}&\,,
\ea
$$
while the remainder of $g^{HH_N}_n$ is estimated by
$$
\ba
\sum\limits_{n>n_0}\left\|\int_0^t dt_1\int_0^{t_1}dt_2\dots\int_0^{t_{n-1}}dt_ng^{HH_N}_n(t,t_1,\dots,t_n)\right\|_{\alpha,\beta}&
\\
\le\sum_{n>n_0}\left(\frac{(1+|t|)|t|C_\Phi^0e^\alpha }{(\beta'-\beta_0)}\right)^n\norm{\un f^{in}}_{\alpha,\beta'}
\le\sum_{n>n_0}\left(\frac{(1+|t|)|t|C_\Phi^\beta(t)e^\alpha }{(\beta'-\beta_0)}\right)^n\norm{\un f^{in}}_{\alpha,\beta'}&
\\
\le\sum_{n>n_0}\frac{|t|^n}{(2T)^n}\norm{\un f^{in}}_{\alpha,\beta'}=\frac{(|t|/2T)^{n_0+1}}{1-\frac{|t|}{2T}}\norm{\un f^{in}}_{\alpha,\beta'}&\,.
\ea
$$
The slight difference between both estimates comes from the term $T^N$ in the definition of $\un g^N_n$: see the discussion before (\ref{EstiSigman}), together with \eqref{estC} and \eqref{Tab}.  Using the elementary inequality
\be\lb{ElemIneq}
\|g_j\|_\b\le e^{\a j}\|\un g\|_{\a,\b}\,,
\ee
we arrive at the following estimate for the remainder:
\be\label{preum}
R(n_0)_j\leq \frac{2e^{\alpha j}}{1-\frac{|t|}T}\left(\frac{|t|}T\right)^{n_0+1}\norm {\un f^{in}}_{\alpha,\beta'}\,.
\ee

In order to estimate the sum up to $n_0$, we expand the string ${\dys}^N_n\ec (t, t_1,\dots,t_n)$ as 
\be\label{stringo}
{\dys}^N_n(t,t_1,\dots,t_n)=S(t)\sum\limits_{k=0}^{n}\ \ \ \sum\limits_{\substack{\sigma_1,\dots,\sigma_n\in\{0,1\},\\ \sigma_1+\ldots+\sigma_n=n-k}}O_{\sigma_n}(t_n)\dots O_{\sigma_1}(t_1)\un f^{in}
\ee
where 
$$
O_0(t):=S(-t)T^NS(t)\quad\hbox{ and }O_1(t)=S(-t)C^NS(t)\,.
$$
Each occurence of $\sigma=0$ in the inner sum contributes to the number of operators $T^N$s in the expansion, which is precisely $k$. Notice that the term corresponding to $k=0$ is precisely $g_n^{HH_N}$. 

For $\beta<\beta'$, we set
$$
\hat\beta_k:=\beta+k\frac{\beta'-\beta}n\,,\qquad k=0,\dots,n\,,
$$
so that 
$$
\hat\beta_0=\beta\,,\quad\hat\beta_n=\beta'\,,\quad\hbox{ and }\hat\beta_{k+1}-\hat\beta_k=\frac{\beta'-\beta}n\,.
$$
One has
$$
\ba
({\dys}^N_n-{\dys}_n^{HH_N})_j(t,t_1,\dots,t_n)&
\\
=S(t)\sum\limits_{k=1}^{n}\,\,\,\sum\limits_{\substack{\sigma_1,\dots,\sigma_n\in\{0,1\}\\ \sigma_1+\ldots+\sigma_n=n-k}}(O_{\sigma_n}(t_n))_{j+\sigma_n}(O_{\sigma_{n-1}}(t_{n-1})\dots O_{\sigma_1}(t_1)\un f^{in})_{j+\sigma_n}&\,.
\ea
$$ 
(Indeed, as observed above, the contribution of $k=0$ is $g_n^{HH_N}$.) Using Lemma \ref{propesti} and iterating as in Section \ref{5.1} shows that
$$
\ba
\norm{(\bcb{\dys}^N_{n}\!-\!{\dys}_n^{HH_N})_j(t,t_1,\dots,t_n)}_\beta
\leq\sum\limits_{k=1}^{n}\sum\limits_{\substack{\sigma_1,\dots,\sigma_n\in\{0,1\}\\ \sigma_1+\ldots+\sigma_n=n-k}}\left(\frac{j}N\right)^{1-\sigma_n}\left(\frac{(1+|t|)C_\Phi^{\beta'}(t)}{e(\beta'-\beta)}n\right)&
\\
\times\norm{(O_{\sigma_{n-1}}(t_{n-1})\dots O_{\sigma_1}(t_1)\un f^{in})_{j+\sigma_n}}_{\hat\beta_1}&
\\
\leq\sum\limits_{k=1}^{n}\,\,\,\sum\limits_{\substack{\sigma_1,\dots,\sigma_n\in\{0,1\}\\ \sigma_1+\ldots+\sigma_n=n-k}}\left(\frac{j}N\right)^{1-\sigma_n}\left(\frac{j+\sigma_n}N\right)^{1-\sigma_{n-1}}
	\left(\frac{(1+|t|)C_\Phi^{\beta'}(t)}{e(\beta'-\beta)}n\right)^2&
\\
\times\norm{(O_{\sigma_{n-2}}(t_{n-2})\dots O_{\sigma_1}(t_1)\un f^{in})_{j+\sigma_n+\sigma_{n-1}}}_{\hat\beta_2}&
\\
\leq\sum\limits_{k=1}^{n}\left(\frac{j+n}N\right)^k\left(\sum\limits_{\substack{\sigma_1,\dots,\sigma_n\in\{0,1\}\\ \sigma_1+\ldots+\sigma_n=n-k}}1\right)
\left(\frac{(1+|t|)C_\Phi^{\beta'}(t)}{e(\beta'-\beta)}n\right)^n\norm{ f^{in}_{j+\sigma_n+\sigma_{n-1}+\dots+\sigma_1}}_{\beta'}&
\\
\leq e^{\alpha (j+n)}\left(\frac{(1+|t|)C_\Phi^{\beta'}(t)}{e(\beta'-\beta)}n\right)^n\sum_{k=1}^n\left(\frac{j+n}N\right)^k{n\choose k}\norm{\un f^{in}}_{\alpha,\beta'}&
\\
=e^{\alpha j}\left(\frac{(1+|t|)C_\Phi^{\beta'}(t)e^\alpha}{e(\beta'-\beta)}n\right)^n\left(\left(1+\frac{j+n}N\right)^n-1\right)\norm{\un f^{in}}_{\alpha,\beta'}&\,.
\ea
$$
Here we have used the obvious inequalities \eqref{ElemIneq} and $\sigma_n+\sigma_{n-1}+\dots+\sigma_1\leq n$, and the binomial formula for the last equality.

For $0\le n\leq n_0$, the mean value inequality implies that
$$
\left(1+\frac{j+n}N\right)^n-1\leq n\,\frac{j+n}N\left(1+\frac{j+n_0}{N}\right)^{n_0-1}\,;
$$ 
besides
\be\label{cn0}
\ba
\left(1+\frac{j+n_0}{N}\right)^{n_0-1}&=\exp\left((n_0-1)\log\left(1+\frac{j+n_0}{N}\right)\right)
\\
&\leq\exp\left(\frac{(n_0-1)(j+n_0)}{N}\right)=:c_{n_0}.
\ea
\ee

Therefore, using again \eqref{tbeta},
$$
\ba
\left\|\sum_{n=0}^{n_0}\int_0^tdt_n\int_0^{t_n}dt_{n-1}\dots\int_0^{t_{2}}dt_1\left({\dys}^N_n-{\dys}_n^{HH_N}\right)_j(t,t_1,\dots,t_n))\right\|_{\beta}
\\
\leq c_{n_0}e^{\alpha j}\sum^{n_0}_{n=1}\frac1{n!}\left(|t|\frac{(1+|t|)C_\Phi^{\beta'}(t)e^\alpha n}{e(\beta'-\beta)}\right)^n\frac{j+n}Nn\norm{\un f^{in}}_{\alpha,\beta'}
\\
\leq c_{n_0}\frac{e^{\alpha j}}N\sum^{\infty}_{n=1}\left(|t|\frac{(1+|t|)C_\Phi^{\beta'}(t)e^\alpha}{(\beta'-\beta)}\right)^n({j+n})n\norm{\un f^{in}}_{\alpha,\beta'}
\\
\le c_{n_0}\frac{e^{\alpha j}}N\left(\frac{(j+1)\frac{|t|}T}{(1-\frac{|t|}T)^2}+\frac{2(\frac{|t|}T)^2}{(1-\frac{|t|}T)^3}\right)\norm{\un f^{in}}_{\alpha,\beta'}
\\
=\frac{|t|}Te^{\alpha j}\left(c_{n_0}\frac{c_1}N\right)\norm{\un f^{in}}_{\alpha,\beta'},
\ea
$$
with
\be\label{C1}
c_1:=\frac {j+1}{(1-{\frac{|t|}T})^2}+\frac{2|t|/T}{(1-{\frac{|t|}T})^3}
\ee
and $c_{n_0}$ defined by \eqref{cn0}. Eventually, we arrive at the bound
\be\lb{mko}
\ba
\left\|\sum\limits_{n=0}^\infty\int_0^tdt_n\int_0^{t_n}dt_{n-1}\dots\int_0^{t_{2}}dt_1\left({\dys}^N_n-{\dys}_n^{HH_N}\right)_j(t,t_1,\dots,t_n))\right\|_{\beta}&
 \\
 \le\frac{|t|}Te^{\alpha j} \left(c_{n_0}\frac{c_1}N+\frac{2\left({|t|}/T\right)^{n_0}}{1-\frac{|t|}T}\right)\norm{\un f^{in}}_{\alpha,\beta'}&\,.
\ea
\ee

\subsection{Comparing $\bcb {\dys}_n\ec$ and ${\dys}_n^{HH_N}$}

We first notice that 
$$
\ba
({\dys}^{HH_N}_n(t,t_1,\dots,t_n))_j&
\\
=S(t)S(-t_1)C^N_{j+1}S(t_1)S(-t_2)C^N_{j+2}S(t_2)\dots S(-t_n)C^N_{j+n}S(t_n)f^{in}_{j+n}&\,,
\ea
$$ 
and  that the same holds for $\dys_n$:
$$
\ba
({\dys}_n(t,t_1,\dots,t_n))_j&
\\
=S(t)S(-t_1)C_{j+1}S(t_1)S(-t_2)C_{j+2}S(t_2)\dots S(-t_n)C_{j+n}S(t_n)f^{in}_{j+n}&\,.
\ea
$$
Since $C^N_{j+1}=(1-\frac jN)C_{j+1}$, one has
$$
({\dys}_n-{\dys}^{HH_N}_n)_j(t,t_1,\dots,t_n)=\left(1-\prod\limits_{l=1}^n\left(1-\frac{j+l}N\right)\right)({\dys}_n(t,t_1,\dots,t_n))_j\,.
$$ 
As before we truncate the summation at $n_0$, and consider the finite sum
$$
\sum_{n=1}^{n_0}\int_0^tdt_n\int_0^{t_n}dt_{n-1}\dots\int_0^{t_{2}}dt_1\left(\bcb {\dys}_n^{}\ec-{\dys}_n^{HH_N}\right)_j(t,t_1,\dots,t_n)\,.
$$
For $j+n_0\leq N$ and $n\leq n_0$,
$$
1-\prod\limits_{l=1}^{n}\left(1-\frac{j+l}N\right)\leq 1-\left(1-\frac{j+n}N\right)^n\leq \frac{n(j+n)}N\,.
$$
Therefore
$$
\ba
\left\|\sum_{n=1}^{n_0}\int_0^tdt_n\int_0^{t_n}dt_{n-1}\dots\int_0^{t_{2}}dt_1\left(\bcb {\dys}_n^{}\ec-{\dys}_n^{HH_N}\right)_j(t,t_1,\dots,t_n)\right\|_{\beta}&
\\
\leq
\sum_{n=1}^{n_0}\int_0^tdt_n\int_0^{t_n}dt_{n-1}\dots\int_0^{t_{2}}dt_1\frac{n(j+n)}N\norm{({\dys}_n^{}(t,t_1,\dots,t_n)_j}_\beta&
\\
\leq
\sum_{n=1}^{n_0}\int_0^tdt_n\int_0^{t_n}dt_{n-1}\dots\int_0^{t_{2}}dt_1\frac{n(j+n)}Ne^{\alpha j}\norm{{\dys}_n^{}(t,t_1,\dots,t_n)}_{\alpha,\beta}&
\\
\le
\sum_{n=1}^\infty \frac{|t|^n}{n!}\frac{n(j+n)}Ne^{\alpha j}\left(\frac{(1+|t|)C_\Phi^0e^\alpha }{e(\beta'-\beta)}n\right)^n\norm{\un f^{in}}_{\alpha,\beta'}&
\\
\leq\sum_{n=1}^\infty\frac{n(j+n)}Ne^{\alpha j}\left(\frac{|t|(1+|t|)C_\Phi^{\beta'}(t)e^\alpha }{\beta'-\beta}\right)^n\norm{\un f^{in}}_{\alpha,\beta'}&
\\
\leq
e^{\alpha j}\sum_{n=1}^\infty\left(\frac{|t|}T\right)^n\frac{n(j+n)}N\norm{\un f^{in}}_{\alpha,\beta'}=\frac{e^{\alpha j}}N\frac{|t|}Tc_1\norm{\un f^{in}}_{\alpha,\beta'}&\,.
\ea
$$
Here we have used \eqref{cxx} and \eqref{tbeta} with $\beta<\beta_0$, together with the inequality
$$
\frac{(1+{|t|}){|t|}C_\Phi^0e^\alpha}{\beta'-\beta}\leq\frac{2(1+|t|)|t|C_\Phi^{\beta'}(t)e^\alpha}{\beta'-\beta_0}\leq\frac{|t|}T\,,
$$
and the definition \eqref{C1} of $c_1$. 

The remainder for the series expansion giving ${\dys}_n-{\dys}^{HH_N}_n$ is controlled as in Section \ref{5.1}: see estimate \eqref{preum}. Hence
\be\lb{mko'}
\ba
\left\|\sum\limits_{n=0}^\infty\int_0^tdt_n\int_0^{t_n}dt_{n-1}\dots\int_0^{t_{2}}dt_1\left(\bcb {\dys}_n^{}\ec-{\dys}_n^{HH_N}\right)_j(t,t_1,\dots,t_n))\right\|_{\beta}&
\\
\leq\frac{|t|}Te^{\alpha j}\left(\frac {c_1} N+\frac{2\left({|t|}/T\right)^{n_0}}{1-\frac{|t|}T}\right)\norm{\un f^{in}}_{\alpha,\beta'}&\,.
\ea
\ee

\subsection{End of the proof of Theorem \ref{main}}\label{endproof}

Choosing now 
$$
n_0\equiv n_0(N):=[\log(N)/\log(T/|t|)]+1\,,
$$ 
we see that
$$
\left(\frac{|t|}T\right)^{n_0}\le\frac1N\,.
$$
Besides
$$
\ba
c_{n_0(N)}=&\exp\left(\frac{(n_0(N)-1)(j+n_0(N))}{N}\right)
\\
\le&\exp\left(\frac{\log(N)((j+1)\log(T/|t|)+\log(N))}{N(\log(T/|t|))^2}\right)\,.
\ea
$$
Elementary computations show that, for each $a>0$, the function 
$$
x\mapsto x(a+x)e^{-x}
$$
reaches its maximum on $[0,+\infty)$ for 
$$
x=x(a):=1+\tfrac12\left(\sqrt{a^2+4}-a\right)\,.
$$
Hence
$$
\ba
\g(j,\tfrac{|t|}T)&:=\exp\left(\frac{(j+1)\log(\frac{T}{|t|})x((j+1)\log(\frac{T}{|t|}))+x((j+1)\log(\frac{T}{|t|}))^2}{e^{x((j+1)\log(\frac{T}{|t|}))}(\log(\frac{T}{|t|}))^2}\right)
\\
&\,\ge c_{n_0(N)}\,.
\ea
$$
Observe that
$$
x(a)=1+\frac1a+O\left(\frac1{a^3}\right)\quad\hbox{ as }a\to+\infty\,.
$$
Hence
$$
\frac{ax(a)a+x(a)^2}{e^{x(a)}}=\frac1e\left(a+3+\frac4a+O\left(\frac1{a^2}\right)\right)
$$
as $a\to+\infty$, and therefore
$$
\g(j,\tau)\sim\exp\left(\frac{(j+1)\log(1/\tau)+3}{e\log(\tau)^2}\right)\quad\hbox{ as }(j+1)\log(1/\tau)\to+\infty
$$
uniformly as $\tau$ runs through compact subsets of $[0,1)$.

\rb{Let us remark that  Theorem \ref{hier} remains obviously true when the  Hartree hierarchy replaces the BBGKY one. Therefore, defining $(\un F^{in}_H)_j=(F^{in})^{\otimes j},\ j=~1,\dots$, so that $\norm{\un F^{in}-\un F^{in}_H}_{\alpha,\beta'}
=
\frac{(e^{-\alpha}\norm{F^{in}}_{\beta'})^{N+1}}{1-e^{-\alpha}\norm{F^{in}}_{\beta'}}$, we define $\un F^N_H(t)$ and $\un F_H(t)$ as the evolutions of 
$\un F^{in}$ and $\un F^{in}_H$ by the Hartree hierarchy flow. 
\\ We get
that 
$\norm{\un F^N_H(t)-\un F_H(t)}_{\alpha,\beta}
\leq
\frac{(e^{-\alpha}\norm{F^{in}}_{\beta'})^{N+1}
}
{({1-\frac{|t|}T})(1-e^{-\alpha}\norm{F^{in}}_{\beta'})}
\leq
\frac{e^{-\alpha}\norm { F^{in}}_{\beta'}}{N(1-\frac{|t|}T)(1-e^{-\alpha}\norm{F^{in}}_{\beta'})}$ 
(remember that $\alpha>\log{\norm{F^{in}}_{\beta'}}$), so that
$\norm{(\un F^N_H)_j(t)-F(t)^{\otimes j}}_\beta\leq \frac{e^{\alpha (j-1)}\norm { F^{in}}_{\beta'}}{N(1-\frac{|t|}T)(1-e^{-\alpha}\norm{F^{in}}_{\beta'})}$.}

Adding \eqref{mko} and \eqref{mko'} and applying the triangle inequality, we find that
$$
\ba
\left\|\sum\limits_{n=0}^\infty\int_0^tdt_n\int_0^{t_n}dt_{n-1}\dots\int_0^{t_{2}}dt_1\left(\bcb {\dys}_n^{N}\ec-{\dys}_n^{}\right)_j(t,t_1,\dots,t_n))\right\|_{\beta}&
\\
\leq\frac{|t|}T\frac{e^{\alpha j}}N\left((1+c_{n_0(N)})c_1+\frac{4}{1-\frac{|t|}T}\right)\norm{\un f^{in}}_{\alpha,\beta'}&\,.
\ea
$$

Therefore (abusing the notation as mentioned in \eqref{AbusNormRho}), we arrive at the inequality
\rb{\be\label{grrr}
\ba
\norm {F^N_{j}(t)-F(t)^{\otimes j}}_{\beta}\leq\norm {f^N_{j}(t)-(\un f^N_H(t))_j}_{\beta}+\norm{(\un F^N_H)_j(t)-F(t)^{\otimes j}}_\beta&
\\
\leq\frac{e^{\alpha j}}N
\frac{|t|}T
\left((1+ c_{n_0})\left(\frac {j+1}{(1-{\frac{|t|}T})^2}+\frac{2|t|/T}{(1-{\frac{|t|}T})^3}\right)+\!\frac{4}{1-\frac{|t|}T}\right)\norm{\un f^{in}}_{\alpha,\beta'}&\\
+\norm{(\un F^N_H)_j(t)-F(t)^{\otimes j}}_\beta&
\\
\leq\frac{e^{\alpha j}}N
\left((1+\gamma(j,t))\left(\frac {j+1}{(1-{\frac{|t|}T})^2}+\frac{2|t|/T}{(1-{\frac{|t|}T})^3}\right)+\!\frac{5}{1-\frac{|t|}T}\right)\frac{e^{-\alpha}\norm{f^{in}}_{\beta'}}{1-e^{-\alpha}\norm{f^{in}}_{\beta'}}&
\\
=\frac{C(j,|t|/T)}N
\frac{e^{\alpha(j-1)}\norm{f^{in}}_{\beta'}}{1-e^{-\alpha}\norm{f^{in}}_{\beta'}}&\,,
\ea
\ee
}
where $1~\leq~ j~\leq~ N$ and
\be\lb{DefCjt}
C(j,\tau):=(1+\gamma(j,\tau))\left(\frac {j+1}{(1-\tau)^2}+\frac{2\tau}{(1-\tau)^3}\right)+\!\frac{\rb{5}}{1-\tau}\,.
\ee 
In particular
$$
C(j,\tau)\sim\frac {j}{(1-\tau)^2}\left(1+\exp\left(\frac{(j+1)\log(1/\tau)+3}{e\log(\tau)^2}\right)\right)
$$
as $j\to\infty$ uniformly as $\tau$ runs through compact subsets of $[0,1)$.

Taking $\alpha=\log{\norm{F^{in}}_{\beta'}}+\log{2}$ and $\tau=|t|/T$ in \eqref{grrr} leads immediately to \eqref{maineq}.

\bigskip
It remains to prove (\ref{nextineq})-\eqref{jjj}. In view of the asymptotic equivalence above for $C(j,\tau)$, one has
$$
\ba
C(j,\tau)\le&\frac {2j}{(1-\tau)^2}\exp\left(\frac{(j+1)\log(1/\tau)+3}{e\log(\tau)^2}\right)(1+\eps(j,\tau))
\\
=&\frac {2D(\tau)j}{(1-\tau)^2}\exp\left(j\left(\log{2}+\frac{1}{e\log(1/\tau)}\right)\right)(1+\eps(j,\tau))
\ea
$$
for $\tau<1$, where
$$
D(\tau):=\exp\left(\frac{\log(1/\tau)+3}{e\log(\tau)^2}\right)
$$
and where $\eps(j,\tau)\to 0$ as $j\to\infty$ uniformly in $\tau$ over compact subsets of $[0,1)$. 

Hence, if $\a'>\a$ and
\be\label{estijgeneral}
1\le j\le\frac{\log N+\log(T/|t|)}{\a'+1/e\log(T/|t|)}
\ee
one has
$$
\frac{C(j,|t|/T)}{N}\frac{|t|}{T}e^{\a j}\le\frac{2D(|t|)}{(1-\frac{|t|}T)^2}\frac{\log(NT/|t|)}{\a'+1/e\log(T/|t|)}\left(\frac{|t|}{TN}\right)^{\frac{\a'-\a}{\a'+1/(e\log(T/|t|))}}(1+\eps(j,t))
$$
and the right hand side of this last inequality vanishes as $j\to\infty$ uniformly in $|t|/T$ over compact subsets  of $[0,1)$, since $D(|t|/T)$ is an increasing function of $|t|/T$. 

Taking $\alpha=\log{\norm{F^{in}}_{\beta'}}+\log{2}$ and $\alpha'=\alpha+\log{2}>\alpha$ with $\tau=|t|/T$ leads immediately to \eqref{nextineq}-\eqref{jjj}.

\section{Final remarks}

Our result, based on the Wigner formalism, holds for a short time only, provided that both the potential and initial data are smooth enough. However, this regularity assumption is used exclusively for obtaining estimates that are independent 
of $\hbar$. 

If one is willing to give up this requirement, a global in time convergence (under much less stringent assumptions) can indeed be recovered along the following lines, in the same spirit as in \cite{Spohn,BGM,Pickl1}. In terms of the $L^1$ 
norm of the Fourier transform of the Wigner function in the $x$ variable only, it is easy to see that both operators $T^N$ and $C^N$ are bounded, with a norm that tends to $\infty$ as $\hbar\to 0$, while $S(t)$ is isometric. Proceeding as  above,
we prove the convergence in the mean-field limit, for a short time which depends only on the size of the norm of $f^{in}$. 

However the procedure can be iterated because uniform bounds (in an arbitrary finite time interval) on the norm of the solutions can be obtained by means of $H_s$ (with $s >d/2$) estimates on the Hartree dynamics.

\smallskip
Moreover, although not needed in the present paper, 
the following estimate improves on \eqref{eqmainhier2}, and may be of independent interest:
\be\label{eqmainhier3}
\norm{\un f^N(t)-\un f^{in}}_{\alpha,\beta}\le|t|\left(\frac{\norm {\un f^{in}}_{\alpha,\beta'}}{\frac{\beta'-\beta}{3(1+|t|) C^{\beta'}_\Phi(t)e^\alpha}-|t|}+\frac{\sum\limits_{l=1}^d\norm {v_l\un f^{in}}_{\alpha,\beta'}}{e(\beta'-\beta)}\right)\,.
\ee

This bound is obtained in the following manner. First, an explicit computation based on the mean value theorem shows that
$$
\norm{\un f^{in}-S(t)\un f^{in}}_{\alpha,\beta}\leq\frac{|t|}{e(\beta'-\beta)}\sum\limits_{l=1}^d\norm{v_l\un f^{in}}_{\alpha,\beta'}\,.
$$
Notice that the second indices on the norms used in both sides of the inequality above are different and satisfy $\beta'>\beta$; observe indeed that
$$
\norm{\d_{{\un x}_j}\un \phi}_{\alpha,\beta}\le\frac{\norm{\un \phi}_{\alpha,\beta'}}{e(\beta'-\beta)},\qquad j=1,\ldots,d,
$$
this inequality is applied to $\un\phi=v_j\un f^{in}$. This account for the second term on the right hand side of \eqref{eqmainhier3}. The first term on the right hand side of \eqref{eqmainhier3} is obtained by applying the geometric series
estimate \eqref{estigeo} to
$$
\un {\dys}^N(t)-S(t)\un f^{in}=\sum\limits_{n\ge 1}\int_0^tdt_n\int_0^{t_n}dt_{n-1}\dots\int_0^{t_{2}}dt_1{\dys}_n(t,t_1,\dots,t_n)\,.
$$
At variance with \eqref{estigeo}, the summation here starts with $n=1$; hence 
$$
\ba
\norm{\un f^N(t)-S(t)\un f^{in}}_{\alpha,\beta}\le\sum_{n=1}^\infty\left(\frac{3(1+|t|)|t|C_\Phi^{\beta'}(t)e^\alpha }{(\beta'-\beta_0)}\right)^n\norm{\un f^{in}}_{\alpha,\beta'}
\\
=\norm{\un f^{in}}_{\alpha,\beta'}\frac{\frac{3(1+|t|)|t|C_\Phi^{\beta'}(t)e^\alpha }{(\beta'-\beta_0)}}{1-\frac{3(1+|t|)|t|C_\Phi^{\beta'}(t)e^\alpha }{(\beta'-\beta_0)}}\,.
\ea
$$

\part{Interpolation}\label{interpolation}


\section{$\hbar$-dependent bound}


In this section we give a proof of the convergence of the \rr{(marginals of the)} $N$-body density operator to the \rr{(tensor powers of the)} solution of the Hartree equation in trace norm. This result does not require any regularity condition 
on the two-body potential but the convergence is nonuniform in the Planck constant. Indeed, we estimate both terms in the commutator separately and add the corresponding bounds, without taking into consideration any compensation that 
might come from the difference and might lead, after dividing by $\hbar$, the desired uniformity in the Planck constant. This procedure is of course not original and there is a considerable amount of literature on this subject, starting with the 
seminal paper of Spohn \cite{Spohn1}.

The  new feature  in the analysis below is that we keep track of the dependence in $\hbar$ of the rate of convergence so obtained in Theorem \ref{comparison}. Indeed, in the next section, the estimate in Theorem \ref{comparison} is
interpolated with the convergence rate obtained in \cite{GMouPaul}  in order to obtain the uniform in $\hbar$ bound stated in Theorem \ref{mainlog}. 

The explicit rate of convergence of the mean-field limit has been discussed in \cite{RS} and \cite {CLS}, among others. For $\hbar=1$ and an initial condition in the form of a pure state with appropriate regularity, one can bound the error by 
$Ce^{Ct}/N$  (see \cite{CLS}). This result involves the formalism of quantum field theory in Fock space with $\hbar=1$ and semiclassical methods using $\frac1N$ as a fictitious Planck constant. But it is not obvious to see how the limiting
processes defined by letting the genuine and the fictitious Planck constants $\hbar$ and $\frac1N$ tend to $0$ may affect each other.

On the other hand, iterating the convergence rate obtained from estimating the BBGKY hierarchy for $\hbar=1$ over short time intervals leads to a bound that is much less sharp. It is interesting to compare the result stated in \cite{SSS}  
with the bound obtained in the forthcoming Theorem \ref{comparison} with $\hbar=1$. 

We could not find in the literature a detailed proof of the convergence rate based on iterating on the short time estimate for the BBGKY hierarchy. However, we need this convergence rate to track precisely the $\hbar$-dependence in the 
mean-field limit ($N\to\infty$). Therefore, we shall present the iteration argument on BBGKY hierarchies in detail, following a strategy used in \cite{PWZ} for a different problem. 

Since we seek a trace norm error estimate, it is more convenient to work directly on density operators rather than with their Wigner functions. Consider the unitary flow $U_j(t)$ defined on $L^2(\bR^{jd})$ by the formula
$$
{i\hbar}\dot U_j(t):=-\tfrac12\hbar^2\Delta_{\bR^{jd}}U_j(t)\,,\qquad U_j(0)=\hbox{Id}_{L^2(\bR^{jd})}\,.
$$
Consider next the operator $\cS_j(t)$ defined on density operators by conjugation with $U_j(t)$:
$$
\cS_j(t)F=U_j(t)FU_j(-t)\,.
$$
Obviously, $\cS_j(t)$ is a linear isometry on $\cL^1(L^2((\bR^{dj})))$, the space of trace-class operators equipped with the trace-norm. 

Next, we recall the definition of $\cC_j$ and $\cT_j$ given by \eqref{SCj} and \eqref{STj}. For notational convenience we also denote
\begin{equation}
\cC^\hbar _j:=\frac{\cC_j}{i \hbar}\,;\qquad\cT^\hbar _j:=\frac{\cT_j}{i \hbar}\,,\qquad 1\le j\le N\,.
\end{equation}

Henceforth we denote by $F^N_j(t)$ the solution of the $N$-particle BBGKY hierarchy, and by $F(t)$ the solution of the Hartree equation. Thus $F_j(t)=F(t)^{\otimes j}$ is a solution of the Hartree hierarchy, and we arrive at the following
convergence estimate, which is obviously not uniform as $\hbar\to 0$.

\begin{Thm}\label{comparison}
Let $\Phi\in L^\infty(\bR^d)$. For each $j\ge 1$, there exists $N_0(j)\ge 1$ such that
$$
\|F^N_j (t)-F(t)^{\otimes j}\|_{\cL_1(L^2(\bR^{dj}))}\leq\frac{2^{j+1+\frac{16t\nPh}{\h}}}{N^{\left(2^{-1-\frac{16t\nPh}{\h}}\log{2}\right)}}
$$
for all $t\geq 0$ and all $N\geq\max\left(N_0(j),\exp\left(2^{\frac{16\nPh t}{\h}+1}j\right)\right)$.
\end{Thm}

\begin{proof}
We first rephrase \eqref {NShier} and \eqref{Hhier} in the following mild form:
\begin{equation}\label{MNShier}
F^N_j(t)=S(t)F_{0,j}+\int_0^td\tau\left(\cS(t-\tau)\frac{\cT_j^{\hbar}}{N}F^N_{j}(\tau)+\frac{N-j}{N}\cS(t-\tau)\cC_{j+1}^\hbar F^N_{j+1}(\tau)\right)\,,
\end{equation}
and
\begin{equation}
\label{MHhier}
F_j(t)=\cS(t)F_{0,j}+\int_0^td\tau\cS(t-\tau)\cC_{j+1}^\hbar F_{j+1}(\tau)\,.
\end{equation}
Here $\{F_{0,j }\}_{j\ge 1}$ denotes the common initial datum for both sequences, assumed to be factorized, i.e. $F_{0,j}=f_0^{\otimes j}$. 

Using the group property $\cS(t)=\cS(t-t_1)\cS(t_1)$ for all $0\leq t_1\leq t$, and splitting the integral on the right hand side of (\ref{MHhier}) into two parts as follows:
$$
\int_0^t=\int_0^{t_1}+\int_{t_1}^t\,, 
$$
we see that
\begin{equation}
\label{MNShierbis}
\ba
F^N_j(t)\!=\!\!S(t\!-\!t_1)F_{j}^N(t_1)\!+\!\!\int_{t_1}^t d\tau\left(\cS(t\!-\!\tau)\frac{\cT_j^{\hbar}}{N}F^N_{j}(\tau)\!+\!\frac {N\!-\!j}{N}\cS(t\!-\!\tau)\cC_{j+1}^\hbar F^N_{j+1}(\tau)\right)\,,
\ea
\end{equation}
and
\begin{equation}
\label{MHhierbis}
F_j(t)=\cS(t-t_1)F_{j}(t_1)+\int_{t_1}^td\tau \cS(t-\tau)\cC_{j+1}^\hbar F_{j+1}(\tau)\,,
\end{equation}
for all $t\ge t_1\ge 0$.

We are interested in bounding the difference
\begin{equation}
\Dlt _j (t):=F^N_j (t)-F_j (t)\,,
\end{equation}
which is recast as
\begin{equation}
\label{diff}
\Dlt _j(t)=\cR_j (t)+\int_{t_1}^tdt_2\cS(t-t_2)\cC_{j+1}^\hbar\Delta_{j+1}(t_2)\,,
\end{equation}
where
\be\label{R}
\ba
\cR_j(t):=\cS(t-t_1)\Dlt_{j}(t_1)&
\\
+\frac1N\int_{t_1}^tdt_2\cS(t-t_2)\cT_{j}^\hbar F^N_{j}(t_2)-\frac{j}N\int_{t_1}^tdt_2\cS(t-t_2)\cC_{j+1}^\hbar F^N_{j+1}(t_2)&\,.
\ea
\ee

The following estimates rephrase Lemma \ref{propesti} in the present context. For each integer $j>0$,
\begin{eqnarray}
\norm{\cS_j(t)F_j}_{\cL^1}&=&\norm{F_j}_{\cL^1}\,,\label{Stf1}
\\
\norm{\cC^\h_{j+1}F_{j+1}}_{\cL^1}&\leq&\frac{2j\norm\Phi_{L^\infty}}\hbar\norm{F_{j+1}}_{\cL^1}\,,\label{Cjbis}
\\
\norm{\cT^\h_jF_{j}}_{\cL^1}&\leq &\frac{j(j-1)\norm\Phi_{L^\infty}}\hbar\norm{F_j}_{\cL^1}\,.\label{Tjbis}
\end{eqnarray}
The first and the third inequalities are obvious consequences of the definitions. The second one is an obvious consequence of the fact that the linear map
$$
\cL^1(L^2(\bR^{d(j+1)}))\ni K\mapsto K_j\in\cL^1(L^2(\bR^{dj}))
$$
satisfies the bound
$$
\Tr_{L^2(\bR^{dj})}|K_j|\le\Tr_{L^2(\bR^{d(j+1)})}|K|\,.
$$
See Lemma 2.1 in \cite{BGM} for a detailed proof of this (not completely obvious) inequality.

As a consequence
\be\label{rest}
\|\cR(t)\|_{\cL^1}\leq\|\Dlt (t_1)\|_{\cL^1}+\frac{3|t-t_1|j^2\nPh}{N\h}\,.
\ee

Finally, we infer from \eqref{diff} and \eqref {rest} that
\be\label{dest}
\|\Dlt_j (t)\|_{\cL^1}\leq\|\Dlt_j (t_1)\|_{\cL^1}+\frac{3j^2}{16N}+\frac j{8T_\hbar}\int_{t_1}^tdt_2\|\Dlt_{j+1}(t_2)\|_{\cL^1}\,.
\ee
Here 
$$
T_\hbar=\frac{\h}{16\nPh}\mbox{ and  }0\leq t-t_1\leq T_\hbar\,.
$$ 
Notice that, with this choice, the Dyson expansions associated to the two hierarchies are absolutely convergent (uniformly in $N$) for any $t \in (t_1, t_1+T_\hbar) $.  However, in contrast with the previous section, we will not consider 
the full Dyson series (or infinite expansion), but only a finite truncation thereof. 

\smallskip
\noindent
\textit{Step 1: a truncated dominating series.}

Substituting $\Dlt_{j+1}$ in the right hand side of (\ref{dest}) by its expression in terms of $\Dlt_{j+1}$ given by (\ref{dest}), and iterating $n$ times this operation, we see that
\be\lb{iter}
\ba
\|\Dlt_j (t)\|_{\cL^1}\leq&\sum\limits_{\ell =0}^n\frac 1{8^\ell}\frac {j(j+1)\cdots(j+\ell-1)}{\ell !}\left(\|\Dlt_{j+\ell}(t_1)\|_{\cL^1}+\frac {3(j+\ell)^2}{16N}\right)
\\
&+\frac1{8^n}\frac{j(j+1)\cdots(j+n-1)}{n!}\cdot 2
\\
\leq&2^{j-1}\sum\limits_{\ell =0}^n\frac1{4^\ell}\|\Dlt_{j+\ell}(t_1)\|_{\cL^1}+\frac142^{j-1}\frac {(j+n)^2}N+\frac{2^j}{4^n}\,.
\ea
\ee
for $0\le t-t_1\leq T_\hbar$, and for all integer $n\ge 0$. Here we have used the identity
$$
\frac{2\nPh T}{\h}=\frac 18\,,
$$
and the inequality 
$$ 
\frac {j(j+1)\cdots(j+\ell-1)}{\ell !}={j+\ell-1\choose \ell} \leq 2^{j-1+\ell}\,,
$$
together with the formula
$$
\sum_{\ell\ge 0}\frac1{4^\ell}=\frac43\,.
$$
We also use the obvious bound  
$$
\|\Dlt^N_j (t)\|_{\cL^1}\leq 2\,,\qquad\hbox{ for all }j\ge 1\hbox{ and }t\ge 0\,,
$$ 
in order to bound the last term on the right hand side in (\ref{iter}).

At this point, we split the time interval $(0,t)$ into small intervals of length $T_\hbar$. Specifically, for $k=1,\dots,[\frac{t}{T_\hbar}]+1$, we set
$$
A_j^k:=\sup_{s \in((k-1)T_\hbar,kT_\hbar]}\|\Dlt_j(s)\|_{\cL^1}\,.
$$
Defining
\be\lb{DefVphi}
\vph (k,N):= 2^{-k} \log N\,,
\ee
we seek to prove that
\be\label{fbound}
A_j^k\leq 2^{j+k- \vph(k,N)}\quad\hbox{ for }1\le j\leq\vph (k,N)\,,\quad N\ge 2\,,\,\,k\ge 1\,.
\ee
The announced convergence estimate will easily follow from (\ref{fbound}).

\smallskip
\noindent
\textit{Step 2: proof of (\ref{fbound})}

We prove (\ref{fbound}) by induction on $k$.

For $k=1$, we set $t_1=0$ and choose $n=[\vph(1,N)]$ in (\ref{iter}). Since $\Dlt_j(0)=0$ for all $j=1,\ldots,N$, one has
$$
A_j^1\le 2^{j-1}\left(\frac{(j+n)^2}{4N}+\frac{2}{4^n}\right)\le 2^{j-1}\left(\frac{n^2}{N}+\frac{2}{4^n}\right)\,,
$$
so that (\ref{fbound}) is satisfied if
$$
\frac{n^2}{N}+\frac{2}{4^n}\le 2^{1-n}\,,
$$
since $n-1<\vph(1,N)$. Here $N=e^{2\vph(1,N)}\ge e^{2n}$, so that (\ref{fbound}) is satisfied provided that
$$
\frac{2^nn^2}{e^{2n}}+\frac{2}{2^n}\le 2\,,\quad\hbox{ for all }n\ge 1\,,
$$
which is implied in turn by the inequality
$$
\frac{2^{n/2}n}{e^n}\le 1\,,\quad\hbox{ for all }n\ge 1\,.
$$
To check this last inequality, observe that, for each $a>0$, one has
$$
\sup_{x>0}\left(xe^{-ax}\right)=\frac1{ae}\,,
$$
so that
$$
\frac{2^{n/2}n}{e^n}\le\frac2{(2-\ln 2)e}<1\,.
$$

Next we assume that the inequality (\ref{fbound}) holds for $k-1$. First observe that
$$
\|\Dlt_{j+\ell}((k-1)T_\hbar)\|_{\cL^1}\le A^{k-1}_{j+\ell}
$$
so that, by choosing $t_1=(k-1)T_\hbar$ in (\ref{iter}), we see that
\be\label{A}
A_j^k \leq 2^{j-1}\sum_{\ell =0}^n\frac1{4^\ell}A^{k-1} _{j+\ell}+\frac142^{j-1}\frac{(j+n)^2}N+\frac{2^j}{4^n}\,.
\ee
Since (\ref{fbound}) holds for $k-1$, we conclude that, for $1\le j,n\le\vph(k,N)$, one has
$$
\ba
A_j^k&\leq 2^{j-1}\sum_{\ell=0}^n\frac1{4^\ell}2^{k-1+j+\ell-\vph(k-1,N)}+\frac{2^{j-1}(j+n)^2}{4N}+\frac{2^j}{4^n}
\\
&\le 2^{j-1+k-1+\vph(k,N)-\vph(k-1,N)}\sum_{\ell=0}^n\frac1{2^\ell}+\frac{2^{j-1}(j+n)^2}{4N}+\frac{2^j}{4^n}
\\
&\le 2^{j+k-1+\vph(k,N)-\vph(k-1,N)}+\frac{2^{j-1}(j+n)^2}{4N}+\frac{2^j}{4^n}\,,
\ea
$$
since $\vph(k-1,N)=2\vph(k,N)$. Hence
$$
A_j^k\le\tfrac12\cdot 2^{j+k-\vph(k,N)}+\frac{2^{j-1}(j+n)^2}{4N}+\frac{2^j}{4^n}
$$
if $1\le j,n\le\vph(k,N)$. It remains to prove the existence of $N_0(j)$ such that 
$$
N\ge N_0(j)\Rightarrow\frac{2^{j-1}(j+n)^2}{4N}+\frac{2^j}{4^n}\le\tfrac12\cdot 2^{j+k-\vph(k,N)}
$$
for all $k\ge 1$. Equivalently, one seeks to prove the existence of $N_0(j)$ such that 
$$
N\ge N_0(j)\Rightarrow\frac{(j+n)^2}{4N}+\frac2{4^n}\le 2^{k-\vph(k,N)}\,,
$$
for some appropriate choice of $n$ in terms of $N$ and $k$ to be made later.

Observing that
$$
k\ge 1\Rightarrow 2^{k-\vph(k,N)}\ge 2^{1-\vph(1,N)}\,,
$$
it is obviously enough to find $N_0(j)$ such that
$$
N\ge N_0(j)\Rightarrow\frac{(j+n)^2}{4N}+\frac2{4^n}\le 2^{1-\vph(1,N)}\,.
$$
So far $n$ has been kept arbitrary; choose now $n=[\vph(1,N)]$, so that
$$
2^{\vph(1,N)-1}\left(\frac{(j+n)^2}{4N}+\frac2{4^n}\right)\le 2^n\left(\frac{(j+n)^2}{4e^{2n}}+\frac2{4^n}\right)\to 0
$$
for each $j\ge 1$ as $n\to+\infty$. Hence, for each $j\ge 1$, there exists $n_0(j)\ge 1$ such that 
$$
n\ge n_0(j)\Rightarrow 2^n\left(\frac{(j+n)^2}{4e^{2n}}+\frac2{4^n}\right)\le 1\,.
$$
Thus
$$
N\ge N_0(j):=e^{2n_0(j)+2}\Rightarrow \left(\frac{(j+n)^2}{4N}+\frac2{4^n}\right)\le 2^{1-\vph(1,N)}\le 2^{k-\vph(k,N)}\,.
$$
We conclude that (\ref{fbound}) holds for all $k\ge 1$.

\smallskip
\noindent
\textit{Step 3: end of the proof of Theorem \ref{comparison}.}

At this point, we observe that 
$$
k=[16\|\Phi\|_{L^\infty}t/\hbar]+1\,,
$$
so that
$$
\|\Dlt_j(s)\|_{\cL^1}\le\frac{2^{j+1+\frac{16\|\Phi\|_{L^\infty}t}{\hbar}}}{2^{2^{-1-\frac{16\|\Phi\|_{L^\infty}t}{\hbar}}\log N}}=\frac{2^{j+1+\frac{16\|\Phi\|_{L^\infty}t}{\hbar}}}{N^{2^{-1-\frac{16\|\Phi\|_{L^\infty}t}{\hbar}}\log 2}}\,,
$$
provided that $N\ge N_0(j)$ and $\vph(k,N)\ge j\ge 1$, i.e. $\log N\ge 2^kj$, which is in turn implied by the condition
$$
\log N\ge 2^{1+\frac{16\|\Phi\|_{L^\infty}t}{\hbar}}j\,.
$$
This completes the proof of Theorem \ref{comparison}.
\end{proof}


\section{Proof of Theorem \ref{mainlog}}\label{prooftheo}


We begin with the following elementary observation.

\begin{Lem}
For each trace-class operator $\rho$ on $L^2(\bR^d)$, one has
$$
\norm{\widetilde W_\hbar[\rho]}_{L^1(\bR^d\times\bR^d)}\leq\norm\rho_{\mathcal L^1(L^2(\bR^d))}\,.
$$
\end{Lem}

\begin{proof}
For $z=(q,p)\in\bR^d\times\bR^d$, one has
$$
\widetilde W_\hbar[\rho](z)=(2\pi\hbar)^{-d}\la z |\rho|z\ra=(2\pi\hbar)^{-d}\Tr(\rho|z\ra\la z|)\,,
$$
so that
$$
(2\pi\hbar)^{d}|\widetilde W_\hbar[\rho](z)|=|\Tr(\rho|z\ra\la z|)|\le\Tr|\rho||z\ra\la z|
$$
since $|\rho|\pm\rho\ge 0$ and $|z\ra\la z|\ge 0$. Therefore
$$
\norm{\widetilde W_\hbar[\rho]}_{L^1(\bR^d\times\bR^d)}\leq\Tr\left(|\rho|(2\pi\hbar)^{-d}\int_{\bR^d\times\bR^d}|z\ra\la z|dz\right)=\norm\rho_{\mathcal L^1(L^2(\bR^d))}\,,
$$
since
$$
(2\pi\hbar)^{-d}\int_{\bR^d\times\bR^d}|z\ra\la z|dz=\Op^T_\hbar[1]=\hbox{Id}_{L^2(\bR^d)}\,.
$$
\end{proof}

\bigskip
\noindent
\textit{Step 1: the two fundamental estimates}

\smallskip
Therefore, applying Theorem \ref{comparison} implies that
\be\lb{EstimNU}
\|\widetilde W_\hbar[F_j^N(t)]-\widetilde W_\hbar[F(t)^{\otimes j}]\|_{L^1}\leq\frac{2^{j+1+\frac{16t\nPh}{\h}}}{N^{\left(2^{-\frac{16t\nPh}{\h}-1}\log{2}\right)}}\,,
\ee
for all $t\geq 0$ and all $N\geq\max\left(N_0(j),\exp\left(2^{\frac{16\nPh t}{\h}+1}j\right)\right)$.

\smallskip
Assume that $F^{in}$ is a T\"oplitz operator. Then one knows from \cite{GMouPaul} \rb{(using formula (17) of Theorem 2.4 coupled with point (2) of Theorem 2.3)} that
\be\lb{EstimU}
\MKd(\widetilde W_\hbar[F_j^N(t)],\widetilde W_\hbar[F(t)^{\otimes j}])^2\leq j\left(2d\hbar+\frac{C}N\right)e^{\Lambda t}+2jd\hbar\,,
\ee
where 
$$
\Lambda:=3+4\Lip(\grad V)^2\quad\hbox{ and }C:=\frac{8\norm{\nabla\Phi}_{L\infty}}\Lambda\,,
$$ 
while $\MKd$ is the Wasserstein distance of exponent $2$. We recall that, for each pair of Borel probability measures on $\bR^d$ with finite second order moments,
\be\lb{DefMKd}
\MKd(\mu,\nu):=\inf_{\pi\in\Pi(\mu,\nu)}\sqrt{\iint_{((\bR^{d}\times\bR^{d})^j)^2}|z-z'|^2\pi(dzdz')}\,,
\ee
where $\Pi(\mu,\nu)$ is the set of couplings of $\mu$ and $\nu$ defined before \eqref{DefMKo} --- see chapter 7 in \cite{VillaniTOT}.

\smallskip
\noindent
\textit{Step 2: an alternative}

\smallskip
Define
\be\lb{DefE}
\ba
E(N,\hbar,j,t)&
\\
:=\min\left(\norm{\widetilde W_\hbar[F_j^N(t)]-\widetilde W_\hbar[F(t)^{\otimes j}]}_{L^1}^2,\MKd(\widetilde W_\hbar[F_j^N(t)],\widetilde W_\hbar[F(t)^{\otimes j}])^2\right)&\,.
\ea
\ee
Thus, for all $t\geq 0$ and all $N\geq\max\left(N_0(j),\exp\left(2^{\frac{16\nPh t}{\h}+1}j\right)\right)$, one has
\be\lb{IneqE<1}
\ba
E(N,\hbar,j,t)\le\min\left(\frac{2^{2j+2+\frac{32t\nPh}{\h}}}{N^{\left(2^{-\frac{16t\nPh}{\h}}\log2\right)}},j\left(2d\hbar+\frac{C}N\right)e^{\Lambda t}+2jd\hbar\right)&
\\
\le A(N,\hbar,j,t)+\frac{jC}Ne^{\Lambda t}&\,,
\ea
\ee
with
\be\lb{DefAB}
\left\{
\ba
{}&A(N,\hbar,j,t):=\min\left(\frac{2^{2j+2+\frac{32t\nPh}{\h}}}{N^{\left(2^{-1-\frac{16t\nPh}{\h}}\log2\right)}},B(\hbar,j,t)\right)\,,
\\	\\
&B(\hbar,j,t):=2jd\hbar(1+e^{\Lambda t})\,.
\ea
\right.
\ee
Equivalently, the inequality (\ref{IneqE<1}) holds if 
$$
N\ge N_0(j)\quad\hbox{ and }\quad\hbar\ge\frac{16\nPh t\ln 2}{\log\log N-\log(2j)}=:\hat\hbar(N,j,t)\,.
$$

\smallskip
Observe now that the function $t\mapsto A(N,\hbar,j,t)$ is increasing, since it is the min of two increasing functions of $t$. Thus, for each $t^*>0$, each $t\in[0,t^*]$, each integer $j\ge 1$ and each $N\ge N_0(j)$ one has the following 
alternative:

\smallskip
\noindent
(a) either $\hbar\ge\hat\hbar(N,j,t)$, in which case,
$$
\ba
E(N,\hbar,j,t)\le A(N,\hbar,j,t)+\frac{jC}Ne^{\Lambda t}\le A(N,\hbar,j,t^*)+\frac{jC}Ne^{\Lambda t^*}&\,,
\ea
$$
(b) or $\hbar<\hat\hbar(N,j,t)$ and
$$
\ba
E(N,\hbar,j,t)\le\MKd(\widetilde W_{F_j^N(t)},\widetilde W_{F(t)^{\otimes j}})^2\le B(\hbar,j,t)+\frac{jC}Ne^{\Lambda t}&
\\
\le B(\hat\hbar(N,j,t^*),j,t^*)+\frac{jC}Ne^{\Lambda t^*}&\,.
\ea
$$

\smallskip
\noindent
\textit{Step 3: the global in $\hbar$ inequality}

\smallskip
Since the first and the second arguments in the min on the right hand side of (\ref{DefE}) are respectively decreasing and increasing functions of $\hbar$, one has
$$
A(N,\hbar,j,t^*)\le B(\hbar,j,t^*)\Bigg|_{\frac{2^{2j+2+\frac{32t^*\nPh}{\h}}}{N^{\left(2^{-\frac {16t^*\nPh}{\h}}\log2\right)}}=2jd\hbar(1+e^{\Lambda t^*})}\,.
$$
Observe that
$$
\ba
\frac{2^{2j+2+\frac{32t^*\nPh}{\h}}}{N^{\left(2^{-\frac{16 t^*\nPh}{\h}}\log2\right)}}&=2jd\hbar(1+e^{\Lambda t^*})
\\
&\Leftrightarrow
\frac{2^{2j+2+\frac{32t^*\nPh}{\h}}}{2jd\hbar(1+e^{\Lambda t^*})}=N^{\left(2^{-\frac{16t^*\nPh}{\h}}\log2\right)}
\\
&\Leftrightarrow
\left(\frac{2^{2j+2+\frac{32t^*\nPh}{\h}}}{2jd\hbar(1+e^{\Lambda t^*})}\right)^{2^{\frac {16t^*\nPh}{\h}}/\log{2}}=N
\ea
$$
is an equation for $\hbar$ in terms of $N$ with only one solution $\hbar(N,j,t^*)$, since the left hand side of the last equality is a continuous function of $\hbar$ decreasing from $+\infty$ to $0$ as $\hbar$ runs through $(0,+\infty)$.
Hence
$$
A(N,\hbar,j,t^*)\le B(\hbar(N,j,t^*),j,t^*)
$$
and we conclude from the alternative (a)-(b) above that, for each integer $j\ge 1$ and each $t^*>0$, one has the uniform in $\hbar$ bound
\be\lb{IneqEU}
\sup_{\hbar>0}\sup_{0\le t\le t^*}E(N,\hbar,j,t)\le\max(B(\hbar(N,j,t^*),j,t^*),B(\hat\hbar(N,j,t^*),j,t^*))+\frac{jC}Ne^{\Lambda t^*}\,.
\ee

\smallskip
\noindent
\textit{Step 4: an asymptotic equivalent for $\hbar(N,j,t^*)$}
We find that
$$
\log{N}\!=\!\frac{2^{\frac{16t^*\nPh}{\hbar(N,j,t^*)}}}{\log 2}\left(\left(2j\!+\!2+\!\frac{32t^*\nPh}{\hbar(N,j,t^*)}\right)\log 2\!-\!\log\left(2jd\hbar(N,j,t^*)(1\!+\!e^{\Lambda t^*})\right)\right),
$$
and therefore
$$
\ba
\log{\log{N}}=\frac{16t^*\nPh}{\hbar(N,j,t^*)}\log 2-\log\log 2&
\\
+\log\left(\left(2j+2+\frac{32t^*\nPh}{\hbar(N,j,t^*)}\right)\log2-\log\left(2jd\hbar(N,j,t^*)(1+e^{\Lambda t^*})\right)\right)&\,.
\ea
$$
Hence, in the limit as $N\to\infty$ with $t^*$ and $j$ kept fixed (which implies in particular that $N\ge N_0(j)$), one has
$$
\hbar(N,j,t^*)\sim\frac{16t^*\norm{\Phi}_{L^\infty}\log 2}{\log\log N}\sim\hat\hbar(N,j,t^*)\,.
$$

Since
$$
\frac{jC}Ne^{\Lambda t^*}=o\left(\frac1{\log\log N}\right)\,,
$$
for each $t^*,j$ fixed as $N\to\infty$, inequality (\ref{IneqEU}) implies that
\be\lb{IneqE2}
\ba
\sup_{\hbar>0}\sup_{0\le t\le t^*}E(N,\hbar,j,t)\lesssim B\left(\frac{16t^*\norm{\Phi}_{L^\infty}\log 2}{\log\log N},j,t^*\right)
\\
=\frac{32jdt^*\norm{\Phi}_{L^\infty}(1+e^{\Lambda t})\log 2}{\log\log N}\,,
\ea
\ee
for each $t^*,j$ fixed as $N\to\infty$.

\smallskip
The proof of Theorem \ref{mainlog} follows from (\ref{IneqE2}) and the following observation.

\begin{Lem}\label{last}
Let $\mu$ and $\nu$ be two Borel probability measures absolutely continuous with respect to the Lebesgue measure. Then
$$
\MKo(\mu,\nu)\leq\inf{(\norm{\mu-\nu}_{L^1},\MKd(\mu,\nu))}\,.
$$
\end{Lem}

\begin{proof}
Let $D_1(z,z')=\min{(|z-z'|,1)}$. Obviously $D_1(z,z')^2\leq |z-z'|^2$ so that, for each Borel probability measure $\pi$,
$$
\ba
\left(\iint\pi(dz,dz')D_1(z,z')\right)^2\leq&\iint\pi(dz,dz')\iint\pi(dz,dz')|z-z'|^2
\\
=&\iint\pi(dz,dz')|z-z'|^2\,.
\ea
$$ 
In particular, specializing the inequality above to couplings of $\mu,\nu$ implies that
$$
\MKo(\mu,\nu)\leq\MKd(\mu,\nu)\,.
$$

Consider now the particular coupling $\tilde\pi$ defined as follows:
$$
\tilde\pi:=
\left\{
\ba
{}&\mu(z)\delta(z-z')&\hbox{ if }\mu=\nu\,,
\\
&\lambda(z)\delta(z-z')+\frac{(\mu(z)-\lambda(z))(\nu(z')-\lambda(z'))}{\displaystyle 1-\int\lambda(dz'')}&\hbox{ if }\mu\not=\nu\,,
\ea
\right.
$$
where $\lambda=\min(\mu,\nu)$. (We have abused the notation, and designated the probability measures $\mu,\nu$ and their densities with respect to the Lebesgue mesaure by the same letter.) 

Thus
$$
\MKo(\mu,\nu)\leq\int\pi(dz,dz')\min{(|z-z'|,1)}\leq 1-\int\lambda(dz)\leq\norm{\mu-\nu}_{L^1}\,,
$$
and this completes the proof.
\end{proof}

\begin{appendix}

\section{Proof of Lemma \ref{oufoufouf}}\label{prouf}


Taking the Fourier transform of both sides of \eqref{Nw}, we easily arrive at the following equality:
$$
\ba
\partial_t\widetilde{f^N}(t,\un\xi_N;\un\eta_N)+\un\xi_N\cdot\nabla_{\un\eta_N}\widetilde{f^N}(t,\un\xi_N;\un\eta_N)&
\\
=\int_{\bR^{Nd}}{d\un h_N}{}\widehat{V_N}(\un h_N)\frac{2\sin{(\frac\hbar2\un\eta_N\cdot \un h_N)}}\hbar\widetilde{f^N}(t,\un \xi_N+\un h_N;\un \eta_N)&\,.
\ea
$$
Since 
$$
V_N(\un x_N):=\frac1{2N}\sum\limits_{1\le l\neq r\le N}\Phi(x_l-x_r)\,,
$$
one has
$$
\ba
\widehat{V_N}(\un h_N)&=\tfrac1{(2\pi)^{dN}}\int_{\bR^{dN}}V_N(\un x_N)e^{-i\un x_N\cdot\un h_N}d\un x_N
\\
&=\frac{1}{2N}\sum_{l\neq r=1}^N\widehat\Phi(h_l)\delta(h_l+h_r)\prod_{k\notin\{l,r\}}\delta(h_k)\,,
\ea
$$
with $\un h_N=(h_1,\dots,h_N)$. Therefore, according to \eqref{defh},
\be\label{Nj}
\ba
\int_{\bR^{Nd}} {d\un h_N}{}\widehat{V_N}(\un h_N)\frac{2\sin{(\frac\hbar2\un\eta_N\cdot \un h_N)}}\hbar\widetilde{f^N}(t,\un \xi_N+\un h_N;\un \eta_N)&
\\
=\frac1{2N}\sum_{1\le l\neq r\le N}\int_{\bR^d}{dh}{}\widehat\Phi(h)\frac{\sin{(\frac\hbar2(\eta_r-\eta_l)\cdot h)}}\hbar\widetilde{f^N}(t,\un \xi_N+\un h^r_N-\un h^l_N;\un \eta_N)&\,.
\ea
\ee
Observe that
$$
\widetilde{f^N_j}(\un\xi_j,\un\eta_j)=\widetilde{f^N}(\un\xi_j,0_{N-j};\un\eta_j,0_{N-j}),
$$
Moreover
$$
\ba
\int_{\bR^{Nd}} {d\un h_N}{}\widehat{V_N}(\un h_N)\frac{2\sin{(\frac\hbar2(\un\eta_j,0_{N-j})\cdot \un h_N)}}\hbar\widetilde{f^N}(t,(\un \xi_j,0_{N-j})-\un h_N;(\un \eta_j,0_{N-j}))
\\
=
\frac1N\widetilde{T_j}\widetilde{f_j^N}(t,\un\xi_j;\un\eta_j)
+
\frac{N-j}N\widetilde{C_{j+1}}\widetilde{f_{j+1}^N}(t,\un\xi_j;\un\eta_j)
\ea
$$
with
\be\label{defT}
\widetilde{T_j}\widetilde{f_j^N}(t,\un\xi_j;\un\eta_j)=\frac1{2}\sum_{l\neq r=1}^j\int_{\bR^d}{dh}{}\widehat\Phi(h)\frac{2\sin({\frac{\hbar}2(\eta_r-\eta_l)\!\cdot\!h})}\hbar\widetilde{f^N_j}(t,\un\xi_j+\un h^r_j-\un h^l_j;\un\eta_j)\,,
\ee
and
\be\label{defC}
\ba
\widetilde{C_{j+1}}\widetilde{f_{j+1}^N}(t,\un\xi_j;\un\eta_j)&
\\
=
\sum_{r=1}^j\int_{\bR^d}{dh}{}\widehat\Phi(h)\frac{2\sin(\frac{\hbar}{2}\eta_r\!\cdot\!h)}\hbar\widetilde{f_{j+1}^N}(t,(\un\xi_j+\un h^r_j,-h);(\un\eta_j,0))&\,.
\ea
\ee
Indeed the contribution given by $l,r\leq j$ in \eqref{Nj} is precisely $\frac1N$ times the right hand-side of \eqref{defT} by the fact that $\un h^m_N=\un h^m_j$ if $m\leq j$, so that 
$$
\widetilde{f^N}(t,(\un \xi_j,0_{N-j})+\un h^l_N-\un h^r_N;\un \eta_j,0_{N-j})=\widetilde{f^N_j}(t,(\un \xi_j+\un h^l_j-\un h^r_j;\un \eta_j)\,.
$$
Moreover the contribution given by $j< l,r\leq N$ in \eqref{Nj} vanishes because $\eta_l,\eta_r=0$ for  $\un\eta_N=\un\eta_j+0_{N-j}$ and $l,r>j$.

The terms $1\leq r\leq j<l\leq N$ and $1\leq l\leq j<r\leq N$ are the only ones remaining in the sum \eqref{Nj}. The first case will give a contribution to \eqref{Nj} which is
$$
\frac1{2N}\sum_{1\leq r\leq j<l\leq N}\int{dh}{}\widehat\Phi(h)\frac{2\sin(\frac{\hbar}{2}\eta_r\!\cdot\!h)}\hbar\widetilde{f_{}^N}(t,(\un\xi_j,0_{N-j})+\un h^r_N-\un h^l_j;\un\eta_j,0_{N-j})\,.
$$
Since $(\un\xi_j+0_{N-j})+\un h^r_N-\un h^l_N=(\un\xi_j+\un h^r_j,0_{N-j})-\un h^l_N$, one has
$$
\ba
\widetilde{f^N}(t,(\un\xi_j+\un h^r_j,0_{N-j})-\un h^l_N;\eta_j,0_{N-j})
=&
\widetilde{f^N}(t,(\un\xi_j+\un h^r_j,-h,0_{N-j-1});\eta_j,0_{N-j})
\\
=&
\widetilde{f_{j+1}^N}(t,\un\xi_j+\un h^r_j,-h;\un\eta_j,0)&\,.
\ea
$$
because $f^N(t,\cdot;\cdot)$ is symmetric.

Therefore
$$
\ba
\frac1{2N}\sum_{1\leq r\leq j<l\leq N}\int{dh}{}\widehat\Phi(h)\frac{2\sin(\frac{\hbar}{2}\eta_r\!\cdot\!h)}\hbar\widetilde{f_{}^N}(t,(\un\xi_j,0_{N-j})+\un h^r_N-\un h^l_N;(\un\eta_j,0_{N-j}))&
\\
=
\frac{N-j}{2N}\sum_{r=1}^j\widehat\Phi(h)\frac{2\sin(\frac{\hbar}{2}\eta_r\!\cdot\!h)}\hbar\widetilde{f_{j+1}^N}(t,(\un\xi_j+\un h^r_j,-h);(\un\eta_j,0))&\,.
\ea
$$
It is easy to check that the contribution coming from $1\leq l\leq j<r\leq N$ in \eqref{Nj} gives the same expression (as can be seen by using the substitution $h\mapsto-h$), and the sum of the two gives \eqref{defC}.

Finally we get that equation \eqref{Nw} is equivalent to
\be\label{kjhg}
\ba
\partial_t\widetilde{f^N_j}(t,\un\xi_j,\un\eta_j)+\un\xi_j\cdot\nabla_{\un\eta_j}\widetilde{f^N_j}(t,\un\xi_j,\un\eta_j)&
\\	\\
=
\frac{N-j}N\widetilde{C_{j+1}}\widetilde{f_{j+1}^N}(t,\un\xi_j;\un\eta_j)+\frac1N\widetilde{T_j}\widetilde{f_j^N}(t,\un\xi_j;\un\eta_j)&\,,\qquad j=1\dots N\,,
\ea
\ee
and the lemma follows from the formulas $\widetilde{T_jf_j}=\widetilde{T_j}\widetilde{f_j}$ and $\widetilde{C_{j+1}f_{j+1}}=\widetilde{C_{j+1}}\widetilde{f_{j+1}}$ after taking the inverse Fourier transform of both sides of \eqref{kjhg}.
\end{appendix}


\vskip 1cm
\textbf{Acknowledgements}: This work has been partially carried out thanks to the support of the A*MIDEX project (n$^o$ ANR-11-IDEX-0001-02) funded by the ``Investissements d'Avenir" French Government program, managed by the French 
National Research Agency (ANR) \rb{and the LIA AMU-CNRS-ECM-INdAM
Laboratoire Ypatia des Sciences Math\'ematiques (LYSM)}. T.P. thanks also the Dipartimento di Matematica, Sapienza Universit\`a di Roma, for its kind hospitality during the completion of this work.



\begin{thebibliography}{99}

\bibitem{BGM}
C. Bardos, F. Golse, N. Mauser:
\textit{Weak coupling limit of the $N$ particles Schr\"odinger equation},
Methods Appl. Anal. \textbf{7} (2000), no.2, 275--293.

\bibitem{BEGMY}
C. Bardos, L. Erd\"os, F. Golse, N. Mauser, H.-T. Yau:
\textit{Derivation of the Schr\"odinger-Poisson equation from the quantum $N$-body problem},
C. R. Acad. Sci. Paris, S\'er. I  \textbf{334} (2002), 515--520.

\bibitem{BPS}
N. Benedikter, M. Porta‚  B. Schlein: 
\textit{Effective Evolution Equations from Quantum Dynamics}, 
Springer Briefs in Math. Phys. 7, Springer Cham, 2016.
 
\bibitem{BraunHepp}
W. Braun, K. Hepp:
\textit{The Vlasov Dynamics and Its Fluctuations in the $1/N$ Limit of Interacting Classical Particles},
Commun. Math. Phys. \textbf{56} (1977), 101--113.

\bibitem{CLS}L. Chen, J. Oon Lee, B. Schlein:
\textit{Rate of Convergence Towards Hartree Dynamics}, 
J. Stat. Phys. \textbf{144} (2011), 872--903

\bibitem{Dobru}
R. Dobrushin: 
\textit{Vlasov equations}, 
Funct. Anal. Appl. \textbf{13} (1979), 115--123.

\bibitem{EESY}
A. Elgart, L. Erd\"os, B. Schlein, H.-T. Yau:
\textit{Nonlinear Hartree equation as the mean field limit of weakly coupled fermions},
J. Math. Pures Appl. (9) \textbf{83} (2004), 1241--1273.

\bibitem{EY} 
L. Erd\"os, H.-T. Yau:
\textit{Derivation of the nonlinear Schrˆdinger equation from a many body Coulomb system}, 
Adv. Theor. Math. Phys. \textbf{5}(6)  (2001), 1169--1205.

\bibitem{FGS}
J. Fr\"olich, S. Graffi and S. Schwartz: 
\textit{Mean-Field- and Classical Limit of Many-Body Schr\"odinger Dynamics for Bosons},
Commun. Math. Phys. \textbf{271} (2007), 681--697.

\bibitem{GMouPaul}
F. Golse, C. Mouhot, T. Paul:
\textit{On the Mean-Field and Classical Limits of Quantum Mechanics},
Commun. Math. Phys. \textbf{343} (2016), 165--205.

\bibitem{GMouRicci}
F. Golse, C. Mouhot, V. Ricci:
\textit{Empirical measures and Vlasov hierarchies},
Kinetic and Related Models, \textbf{6} (2013), 919--943.

\bibitem{GP}
F. Golse, T. Paul:
\textit{The Schr\"odinger Equation in the Mean-Field  and Semiclassical Regime},
Archive Rational Mech. Anal., \rb{\textbf{223} (2017) 57-94.}

\bibitem{GMP}
S. Graffi, A. Martinez, M. Pulvirenti:
\textit{Mean-field approximation of quantum systems and classical limit};
Math. Models Methods Appl. Sci. \textbf{13} (2003), 59--73.
 
\bibitem{LionsPaul}
P.-L. Lions, T. Paul:
\textit{Sur les mesures de Wigner}, 
Rev. Mat. Iberoamericana \textbf{9} (1993), 553--618.

\bibitem{NarnhoSewell}
H. Narnhofer, G. Sewell:
\textit{Vlasov hydrodynamics of a quantum mechanical model}, 
Commun. Math. Phys. \textbf{79} (1981), 9--24.

\bibitem{NeunWick}
H. Neunzert, J. Wick:
Die Approximation der L\"osung von Integro-Differentialgleichungen durch endliche Punktmengen;
Lecture Notes in Math. vol. 395, 275--290, Springer, Berlin (1974).

\bibitem{PP}
F. Pezzoti, M Pulvirenti: 
\textit{Mean-Field Limit and Semiclassical Expansion of a Quantum Particle System},  
Ann. Henri Poincar\'e \textbf{10} (2009), 145--187.

\bibitem{Pickl1}
P. Pickl:
\textit{A simple derivation of mean-field limits for quantum systems},
Lett. Math. Phys. \textbf{97} (2011), no. 2, 151--164.

\bibitem{PWZ}
M. Pulvirenti, W. Wagner and M.B. Zavelani Rossi:
\textit{Convergence of particle schemes for the Boltzmann equation}, 
Eur. J. Mech. B/Fluids, \textbf{13} 3 (1994),  339--351.

\bibitem{RS} 
I. Rodnianski, B. Schlein:
\textit {Quantum fluctuations and rateof convergence towards mean-field dynamics},
Commun. Math. Phys. \textbf{291}(1) (2009), 31--61. 

\bibitem{Spohn1}
H. Spohn:
\textit{Kinetic equations from Hamiltonian dynamics},
Rev. Mod. Phys. \textbf{52} (1980), no.3, 600--640.

\bibitem{Spohn}
H. Spohn: 
\textit{On the Vlasov hierarchy}, 
Math. Meth. in the Appl. Sci. \textbf{3} (1981), 445--455.

\bibitem{SSS} 
B. Schlein:
\textit{Derivation of Effective Evolution Equations from Microscopic Quantum Dynamics}, 
preprint.

\bibitem{VillaniTOT}
C. Villani:
``Topics in Optimal Transportation'',
American Mathematical Soc, Providence (RI) (2003)

\end{thebibliography}
\end{document}